\documentclass[a4paper, 12pt, reqno]{amsart}

\pdfoutput=1

\usepackage[margin=2.5cm, headsep=0.75cm, footskip=0.75cm]{geometry}
\usepackage{amsmath, amsthm, amssymb, mathtools}
\usepackage[hang, flushmargin]{footmisc}
\usepackage[english]{isodate}
\usepackage{enumitem}
\usepackage{xcolor}
\usepackage{tikz-cd}
\usepackage[colorlinks=true, linkcolor=blue, linkbordercolor=blue, citecolor=red, citebordercolor=red, urlcolor=indigo, linktocpage=true]{hyperref}
\usepackage[capitalise, nameinlink, noabbrev, nosort]{cleveref}
\usepackage{doi}

\definecolor{indigo}{HTML}{492DA5}

\allowdisplaybreaks

\setenumerate{listparindent=\parindent}

\providecommand{\noopsort}[1]{}

\makeatletter\g@addto@macro\bfseries{\boldmath}\makeatother

\makeatletter
\let\origsection\section
\renewcommand\section{\@ifstar{\starsection}{\nostarsection}}
\newcommand\sectionspace{\vspace{0.5ex}}
\newcommand\nostarsection[1]{\sectionspace\origsection{#1}\sectionspace}
\newcommand\starsection[1]{\sectionspace\origsection*{#1}\sectionspace}
\makeatother

\setlist[enumerate]{font=\normalfont}
\crefname{enumi}{}{}
\crefname{enumii}{}{}
\crefname{condition}{condition}{conditions}

\newcommand\numberthis{\addtocounter{equation}{1}\tag{\theequation}}

\numberwithin{equation}{section}
\crefname{equation}{equation}{equations}

\newtheorem{theorem}{Theorem}[section]

\newtheorem{thm}[theorem]{Theorem}
\crefname{thm}{Theorem}{Theorems}

\newtheorem{lemma}[theorem]{Lemma}
\crefname{lemma}{Lemma}{Lemmas}

\newtheorem{prop}[theorem]{Proposition}
\crefname{prop}{Proposition}{Propositions}

\newtheorem{cor}[theorem]{Corollary}
\crefname{cor}{Corollary}{Corollaries}

\theoremstyle{definition}

\newtheorem{definition}[theorem]{Definition}
\crefname{definition}{Definition}{Definitions}

\newtheorem{notation}[theorem]{Notation}
\crefname{notation}{Notation}{Notations}

\theoremstyle{remark}

\newtheorem{remark}[theorem]{Remark}
\crefname{remark}{Remark}{Remarks}

\newtheorem{example}[theorem]{Example}
\crefname{example}{Example}{Examples}

\newcommand{\hl}[1]{\textcolor{magenta}{\emph{#1}}}
\newcommand{\C}{\mathbb{C}}
\newcommand{\N}{\mathbb{N}}
\newcommand{\T}{\mathbb{T}}

\newcommand{\EE}{\mathcal{E}}
\newcommand{\GG}{\mathcal{G}}
\newcommand{\HH}{\mathcal{H}}
\newcommand{\II}{\mathcal{I}}
\newcommand{\MM}{\mathcal{M}}
\renewcommand{\SS}{\mathcal{S}}
\newcommand{\UU}{\mathcal{U}}
\newcommand{\XX}{\mathcal{X}}
\renewcommand{\d}{\mathrm{d}}
\newcommand{\EEo}{\EE^{(0)}}
\newcommand{\IE}{\II^\EE}
\newcommand{\GGo}{\GG^{(0)}}
\newcommand{\GGc}{\GG^{(2)}}
\newcommand{\HHo}{\HH^{(0)}}
\newcommand{\IG}{\II^\GG}
\newcommand{\XE}{\XX^\EE}
\newcommand{\XG}{\XX^\GG}
\newcommand{\GK}{\GG_K}
\newcommand{\res}[1]{\operatorname{res_{#1}}}
\newcommand{\vecspan}{\operatorname{span}}
\newcommand{\supp}{\operatorname{supp}}
\newcommand{\osupp}{\operatorname{osupp}}
\newcommand{\ev}{\operatorname{ev}}
\newcommand{\id}{\operatorname{id}}
\newcommand{\Iso}{\operatorname{Iso}}
\newcommand{\lav}{\ensuremath{\lvert}}
\newcommand{\rav}{\ensuremath{\rvert}}
\newcommand{\lv}{\ensuremath{\lVert}}
\newcommand{\rv}{\ensuremath{\rVert}}

\newcommand{\medcap}{\mathbin{\scalebox{1.2}{\ensuremath{\cap}}}}
\newcommand{\medcup}{\mathbin{\scalebox{1.2}{\ensuremath{\cup}}}}
\newcommand{\restr}[1]{\ensuremath{\vert_{#1}}}

\hypersetup{pdfauthor={Becky Armstrong}, pdftitle={A uniquness theorem for twisted groupoid C*-algebras}}

\date{\today}
\title[A uniqueness theorem for twisted groupoid C*-algebras]{A uniqueness theorem for twisted groupoid C*-algebras}

\author[Becky Armstrong]{Becky Armstrong}
\address[B.\ Armstrong]{Mathematical Institute, WWU M\"unster, Einsteinstr.\ 62, 48149 M\"unster, GERMANY}
\email{\href{mailto:becky.armstrong@uni-muenster.de}{becky.armstrong@uni-muenster.de}}

\thanks{The author would like to thank Nathan Brownlowe and Aidan Sims for many helpful discussions, and for their assistance with preparing and editing this article. The author would also like to thank the anonymous referee for their careful reading and useful suggestions. The author was supported by an Australian Government Research Training Program Stipend Scholarship.}

\begin{document}

\begin{abstract}
We present a uniqueness theorem for the reduced C*-algebra of a twist $\EE$ over a Hausdorff \'etale groupoid $\GG$. We show that the interior $\IE$ of the isotropy of $\EE$ is a twist over the interior $\IG$ of the isotropy of $\GG$, and that the reduced twisted groupoid C*-algebra $C_r^*(\IG;\IE)$ embeds in $C_r^*(\GG;\EE)$. We also investigate the full and reduced twisted C*-algebras of the isotropy groups of $\GG$, and we provide a sufficient condition under which states of (not necessarily unital) C*-algebras have unique state extensions. We use these results to prove our uniqueness theorem, which states that a C*-homomorphism of $C_r^*(\GG;\EE)$ is injective if and only if its restriction to $C_r^*(\IG;\IE)$ is injective. We also show that if $\GG$ is effective, then $C_r^*(\GG;\EE)$ is simple if and only if $\GG$ is minimal.
\end{abstract}

\maketitle

\section{Introduction}

\subsection{Background}

The purpose of this article is prove a uniqueness theorem for the reduced C*-algebra of a twist $\EE$ over a second-countable locally compact Hausdorff \'etale groupoid $\GG$. The study of twisted groupoid C*-algebras was initiated by Renault \cite{Renault1980}, who generalised the construction of twisted group C*-algebras by building full and reduced C*-algebras $C^*(\GG,\sigma)$ and $C_r^*(\GG,\sigma)$ from a second-countable locally compact Hausdorff groupoid $\GG$ admitting a left Haar system and a continuous $2$-cocycle $\sigma$ on $\GG$ taking values in the complex unit circle $\T$. Renault also realised $C^*(\GG,\sigma)$ as a quotient of the C*-algebra associated to the extension of $\GG$ by $\T$ defined by the $2$-cocycle $\sigma$. This construction was subsequently extended by Kumjian \cite{Kumjian1986} to include C*-algebras over groupoid twists that don't necessarily arise from continuous $2$-cocycles.

More recently, Renault \cite{Renault2008} showed that every C*-algebra containing a Cartan subalgebra can be realised as a twisted groupoid C*-algebra, thereby providing a C*-algebraic analogue of Feldman--Moore theory \cite{FM1975, FM1977I, FM1977II}. Renault's reconstruction theorem is of particular importance to the classification program for C*-algebras, given Li's recent article \cite{Li2020} showing that every simple classifiable C*-algebra has a Cartan subalgebra (and is therefore a twisted groupoid C*-algebra), and the work of Barlak and Li \cite{BL2017, BL2020} describing the connections between the UCT problem and Cartan subalgebras in C*-algebras. The increasing interest in twisted groupoid C*-algebras (see, for instance, \cite{ABS2022, AO2020, Boenicke2021, BaH2014, CaH2012, CaHR2013, DGNRW2020, KL2017}) has also recently inspired the introduction of \emph{twisted Steinberg algebras}, which are a purely algebraic analogue of twisted groupoid C*-algebras (see \cite{ACCCLMRSS2021, ACCLMR2022}).

Examples of twisted groupoid C*-algebras include the twisted C*-algebras associated to higher-rank graphs introduced by Kumjian, Pask, and Sims \cite{KPS2012JFA, KPS2013JMAA, KPS2015TAMS, KPS2016JNG}, and the more general class of twisted C*-algebras associated to \emph{topological} higher-rank graphs introduced in the author's PhD thesis \cite{Armstrong2019}.

In this article we prove a uniqueness theorem (\cref{thm: uniqueness theorem}) for the reduced C*-algebra of a twist $\EE$ over a second-countable locally compact Hausdorff \'etale groupoid $\GG$. In particular, we show that the interior $\IE$ of the isotropy of $\EE$ is a twist over the interior $\IG$ of the isotropy of $\GG$, and that a C*-homomorphism of the reduced twisted C*-algebra $C_r^*(\GG;\EE)$ is injective if and only if its restriction to $C_r^*(\IG;\IE)$ is injective. This is an extension of the analogous result \cite[Theorem~3.1(b)]{BNRSW2016} for non-twisted groupoid C*-algebras, and also of the result \cite[Theorem~5.3.14]{Armstrong2019} appearing in the author's PhD thesis for twisted groupoid C*-algebras arising from continuous $2$-cocycles on groupoids. Although many of the arguments used in this article are inspired by their non-twisted counterparts, the twisted setting differs significantly enough from the non-twisted setting to warrant independent treatment. In particular, although $\GG$ is an \'etale groupoid, the twist $\EE$ is not an \'etale groupoid, and this leads to increased technical complexity in many of our proofs. One interesting corollary of our main theorem is that if $\GG$ is effective, then $C_r^*(\GG;\EE)$ is simple if and only if $\GG$ is minimal (see \cref{cor: simplicity}).

\subsection{Outline}

This article is organised as follows. In \cref{sec: prelim} we establish the relevant background and notation, and we recall various well known useful results relating to twists and twisted groupoid C*-algebras. In particular, in \cref{subsec: twists} we recall the definition of a twist $\EE$ over a Hausdorff \'etale groupoid $\GG$, and we show that the interior $\IE$ of the isotropy of $\EE$ is a twist over the interior $\IG$ of the isotropy of $\GG$. In \cref{subsec: twisted groupoid C*s} we recall Kumjian's construction of the full and reduced twisted groupoid C*-algebras $C^*(\GG;\EE)$ and $C_r^*(\GG;\EE)$, and in \cref{prop: C*s of twists and 2-cocycles} we describe the relationship between these C*-algebras and Renault's twisted groupoid C*-algebras arising from continuous $2$-cocycles. This result can be used to translate the results of \cref{sec: twisted isotropy C*s,sec: uniqueness} to the analogous results pertaining to twisted groupoid C*-algebras arising from continuous $2$-cocycles that appear in the author's PhD thesis \cite{Armstrong2019} (see \cref{rem: translations}).

Throughout \cref{sec: uniqueness} we regularly work with twisted C*-algebras associated to the isotropy groups of $\GG$ (which are discrete, since $\GG$ is \'etale), and so in \cref{sec: twisted group C*s} we restrict our attention to twisted C*-algebras associated to discrete groups in order to establish the necessary preliminaries. In particular, we recall the universal property of the full twisted group C*-algebra $C^*(G,\sigma)$ associated to a discrete group $G$ and a $2$-cocycle $\sigma$ on $G$, and in \cref{thm: universal property} we translate this universal property to the language of the associated central extension of $G$ by $\T$.

In \cref{sec: twisted isotropy C*s} we show that the full and reduced twisted C*-algebras of the isotropy groups of $\IG$ are quotients of the full and reduced twisted C*-algebras of the groupoid $\IG$ itself (see \cref{thm: quotient maps}). For the full C*-algebra, we quotient $C^*(\IG;\IE)$ by the ideal generated by functions that vanish on the given isotropy group, but surprisingly, it turns out that this is not the correct ideal to quotient by in the reduced setting. Although the discovery of this fact did not cause us any problems when proving our main theorem, it was somewhat unexpected (at least to the author), and so we provide a proof using an example due to Willett \cite{Willett2015} of a nonamenable groupoid whose full and reduced C*-algebras coincide (see \cref{thm: reduced quotient map not injective}).

A substantial portion of the author's PhD thesis is dedicated to extending well known results of Anderson \cite{Anderson1979} about states of unital C*-algebras to the nonunital setting (see \cite[Section~5.2]{Armstrong2019}). We reproduce some of this material in \cref{sec: states} of the article, as we have been unable to find explicit proofs of these results in the literature, despite them apparently being well known (for instance, they are used in \cite{BNRSW2016}). We apply these results to twisted groupoid C*-algebras in \cref{sec: uniqueness}. The main result of \cref{sec: states} is \cref{thm: unique state extensions}, in which we provide a sufficient \emph{compressibility} condition under which states of (not necessarily unital) C*-algebras have unique state extensions.

In \cref{sec: uniqueness} we observe that there is an embedding $\iota_r$ of $C_r^*(\IG;\IE)$ into $C_r^*(\GG;\EE)$, and that if $\IG$ is closed, then there is a conditional expectation from $C_r^*(\GG;\EE)$ to $\iota_r\big(C_r^*(\IG;\IE)\big)$ extending restriction of functions. We also show that these results hold for the full C*-algebras when $\IG$ is amenable. We then present our main theorem (\cref{thm: uniqueness theorem}), which states that a C*-homomorphism $\Psi$ of $C_r^*(\GG;\EE)$ is injective if and only if the homomorphism $\Psi \circ \iota_r$ of $C_r^*(\IG;\IE)$ is injective. We use this theorem to prove \cref{cor: simplicity}, which states that if $\GG$ is effective, then $C_r^*(\GG;\EE)$ is simple if and only if $\GG$ is minimal. The uniqueness theorem also has potential applications to the study of the ideal structure of twisted C*-algebras associated to Hausdorff \'etale groupoids, and in fact has already been used by the author, Brownlowe, and Sims in \cite{ABS2022} to characterise simplicity of twisted C*-algebras associated to Deaconu--Renault groupoids.

\section{Preliminaries}
\label{sec: prelim}

In this section we present the necessary background on twists over Hausdorff \'etale groupoids and the associated (full and reduced) twisted groupoid C*-algebras. Although groupoid C*-algebras were introduced by Renault in \cite{Renault1980}, we will frequently reference Sims' treatise \cite{Sims2020} on Hausdorff \'etale groupoids and their C*-algebras instead, as it aligns more closely with our setting. The results in this section are presumably well known, but we have presented proofs wherever we have been unable to find them in the literature, or whenever we have felt the need to expand on the level of detail given in existing literature. We begin by recalling some preliminaries on groupoids from \cite[Chapter~8]{Sims2020}.

Throughout this article, $\GG$ will denote a second-countable locally compact Hausdorff groupoid with unit space $\GGo$, which is \hl{\'etale} in the sense that the range and source maps $r, s\colon \GG \to \GGo$ are local homeomorphisms. We refer to such a groupoid as a \hl{Hausdorff \'etale groupoid}, and we denote the set of composable pairs in $\GG$ by $\GGc$. If $\GG$ is \'etale, then $\GG$ admits a Haar system, $\GGo$ is an open subset of $\GG$, and the range, source, and multiplication maps are all open. We call a subset $B$ of $\GG$ a \hl{bisection} if there is an open subset $U$ of $\GG$ containing $B$ such that $r\restr{U}$ and $s\restr{U}$ are homeomorphisms onto open subsets of $\GGo$. Every Hausdorff \'etale groupoid has a (countable) basis of open bisections. Given subsets $A, B \subseteq \GG$, we write $AB \coloneqq \{ \alpha\beta : (\alpha, \beta) \in (A \times B) \cap \GGc \}$ and $A^{-1} \coloneqq \{ \alpha^{-1} : \alpha \in A \}$, and for $\gamma \in \GG$, we write $\gamma A \coloneqq \{\gamma\} A$ and $A \gamma \coloneqq A \{\gamma\}$. Given $u \in \GGo$, we define $\GG_u \coloneqq s^{-1}(u)$, $\GG^u \coloneqq r^{-1}(u)$, and $\GG_u^u \coloneqq \GG_u \cap \GG^u$.\footnote{We acknowledge that this notation is in fact redundant, because using the previously defined notation, we have $\GG_u = \GG u$, $\GG^u = u \GG$, and $\GG_u^u = u \GG u$, for each $u \in \GGo$. However, although the notation that omits the subscripts and superscripts is commonly used in the literature and is arguably more intuitive, we choose not to use it here because we feel that expressions for C*-algebras such as $C_r^*\big(q(u) \GG q(u); u \EE u\big)$ look significantly cleaner when written in the form $C_r^*\big(\GG_{q(u)}^{q(u)}; \EE_u^u\big)$.} For each $u \in \GGo$, the relative topology on $\GG_u$, $\GG^u$, and $\GG_u^u$ is discrete, and $\GG_u^u$ is a countable closed subgroup of $\GG$, called an \hl{isotropy group}. The \hl{isotropy subgroupoid} of $\GG$ is the groupoid $\Iso(\GG) \coloneqq \medcup_{u \in \GGo} \, \GG_u^u = \{ \gamma \in \GG : r(\gamma) = s(\gamma) \}$. We write $\IG$ for the topological interior of $\Iso(\GG)$, and we note that if $\GG$ is a Hausdorff \'etale groupoid, then so is $\IG$. Since $\GGo$ is open in $\GG$, the unit space of $\IG$ is $\GGo$. For each $u \in \GGo$, $\IG_u = \IG \cap \GG_u^u$ is an isotropy group of $\IG$. We say that $\GG$ is \hl{effective} if $\IG = \GGo$. We call a subset $U$ of $\GGo$ \hl{invariant} if $s(\gamma) \in U \implies r(\gamma) \in U$ for all $\gamma \in \GG$, and we say that $\GG$ is \hl{minimal} if $\GGo$ has no nonempty proper open (or, equivalently, closed) invariant subsets.

\subsection{Twists over Hausdorff \'etale groupoids}
\label{subsec: twists}

Groupoid twists and their associated C*-algebras were introduced by Kumjian \cite{Kumjian1986} and subsequently studied by Renault \cite{Renault2008}; however, for consistency of terminology and notation, we will continue to reference Sims' treatise \cite{Sims2020}. We begin by recalling the definition of a twist from \cite[Definition~11.1.1]{Sims2020}.

\begin{definition} \label{def: twist}
A \hl{twist} $(\EE,i,q)$ over a Hausdorff \'etale groupoid $\GG$ is a sequence
\[
\GGo \times \T \overset{i}{\hookrightarrow} \EE \overset{q}\twoheadrightarrow \GG,
\]
where the groupoid $\GGo \times \T$ is viewed as a trivial group bundle with fibres $\T$, $\EE$ is a locally compact Hausdorff groupoid with unit space $\EEo = i\big(\GGo \times \{1\}\big)$, and the following additional conditions hold.
\begin{enumerate}[label=(\alph*)]
\item The maps $i$ and $q$ are continuous groupoid homomorphisms that restrict to homeomorphisms of unit spaces, and we identify $\EEo$ with $\GGo$ via $q\restr{\EEo}$.
\item The sequence is exact, in the sense that $i\big(\{x\} \times \T\big) = q^{-1}(x)$ for each $x \in \GGo$, $i$ is injective, and $q$ is surjective.
\item \label[condition]{cond: local triv} The groupoid $\EE$ is a locally trivial $\GG$-bundle, in the sense that for each $\alpha \in \GG$, there is an open neighbourhood $U_\alpha \subseteq \GG$ of $\alpha$, and a continuous map $P_\alpha\colon U_\alpha \to \EE$ such that
\begin{enumerate}[label=(\roman*), ref={\labelcref{cond: local triv}(\roman*)}]
\item \label[condition]{cond: cts local section} $q \circ P_\alpha = \id_{U_\alpha}$; and
\item \label[condition]{cond: CLS induces homeo} the map $\phi_{P_\alpha}\colon (\beta,z) \mapsto i(r(\beta),z) \, P_\alpha(\beta)$ is a homeomorphism from $U_\alpha \times \T$ onto $q^{-1}(U_\alpha)$.
\end{enumerate}
\item The image of $i$ is central in $\EE$, in the sense that $i(r(\varepsilon),z) \, \varepsilon = \varepsilon \, i(s(\varepsilon),z)$ for all $\varepsilon \in \EE$ and $z \in \T$.
\end{enumerate}
We sometimes denote a twist $(\EE,i,q)$ over $\GG$ simply by $\EE$. We call a continuous map $P_\alpha\colon U_\alpha \to \EE$ satisfying \cref{cond: cts local section} a \hl{(continuous) local section} for $q$, and we call a collection $(U_\alpha, P_\alpha, \phi_{P_\alpha})_{\alpha \in \GG}$ satisfying \cref{cond: local triv} a \hl{local trivialisation} of $\EE$.
\end{definition}

If $\GG$ is a discrete group, then a twist over $\GG$ as defined above is a central extension of $\GG$. It is well known (see, for instance, \cite[Theorem~IV.3.12]{Brown1982}) that there is a one-to-one correspondence between central extensions of a discrete group and $2$-cocycles on the group. This result does not hold in general for groupoids (see \cite[Section~2]{MW1992} or \cite[Section~3]{Boenicke2021}); however, every continuous $\T$-valued $2$-cocycle on a groupoid $\GG$ does give rise to a twist over $\GG$, as we show in \cref{eg: EE_sigma}. To make sense of this example, we first recall the definition of a groupoid $2$-cocycle.

\begin{definition}
A continuous $\T$-valued \hl{$2$-cocycle} on a topological groupoid $\GG$ is a continuous map $\sigma\colon \GGc \to \T$ satisfying
\begin{enumerate}[label=(\roman*)]
\item $\sigma(\alpha,\beta) \, \sigma(\alpha\beta,\gamma) = \sigma(\alpha,\beta\gamma) \, \sigma(\beta,\gamma)$, for all $\alpha,\beta,\gamma \in \GG$ such that $s(\alpha) = r(\beta)$ and $s(\beta) = r(\gamma)$; and
\item $\sigma(r(\gamma),\gamma) = \sigma(\gamma,s(\gamma)) = 1$, for all $\gamma \in \GG$.
\end{enumerate}
\end{definition}

\begin{example} \label{eg: EE_sigma}
Let $\GG$ be a Hausdorff \'etale groupoid and let $\sigma\colon \GGc \to \T$ be a continuous $2$-cocycle. Let $\EE_\sigma = \GG \times_\sigma \T$ be the set $\GG \times \T$ endowed with the product topology. The formulae
\[
(\alpha,w) (\beta,z) \coloneqq \big(\alpha\beta, \, \sigma(\alpha,\beta)wz\big) \quad \text{ and } \quad (\alpha,w)^{-1} \coloneqq \big(\alpha^{-1}, \, \overline{\sigma(\alpha,\alpha^{-1})}\, \overline{w}\big)
\]
define multiplication and inversion operations on $\EE_\sigma$, under which $\EE_\sigma$ is a locally compact Hausdorff groupoid. Let $i_\sigma\colon \GGo \times \T \to \EE_\sigma$ be the inclusion map and let $q_\sigma\colon \EE_\sigma \to \GG$ be the projection onto the first coordinate. Then $(\EE_\sigma, i_\sigma, q_\sigma)$ is a twist over $\GG$.
\end{example}

A routine argument shows that if $(\EE,i,q)$ is a twist over a Hausdorff \'etale groupoid, then the formulae
\[
z \cdot \varepsilon \coloneqq i(r(\varepsilon),z) \, \varepsilon \quad \text{ and } \quad \varepsilon \cdot z \coloneqq \varepsilon \, i(s(\varepsilon),z)
\]
define continuous free left and right actions of $\T$ on $\EE$. Centrality of the image of $i$ implies that $z \cdot \varepsilon = \varepsilon \cdot z$ for all $z \in \T$ and $\varepsilon \in \EE$. This action has the following additional properties.

\begin{lemma} \label{lemma: T-action properties}
Let $(\EE,i,q)$ be a twist over a Hausdorff \'etale groupoid $\GG$ with local trivialisation $(U_\alpha, P_\alpha, \phi_{P_\alpha})_{\alpha \in \GG}$.
\begin{enumerate}[label=(\alph*)]
\item \label{item: z-action is a homeo of EE} For each fixed $z \in \T$, the map $\varphi_z\colon \varepsilon \mapsto z \cdot \varepsilon$ is a homeomorphism of $\EE$.
\item \label{item: q and action} If $\varepsilon, \zeta \in \EE$ satisfy $q(\varepsilon) = q(\zeta)$, then there is a unique $z \in \T$ such that $\varepsilon = z \cdot \zeta$.
\item \label{item: t_alpha cts} For each $\alpha \in \GG$, there is a unique continuous map $t_\alpha\colon q^{-1}(U_\alpha) \to \T$ such that $\phi_{P_\alpha}^{-1}(\varepsilon) = (q(\varepsilon), t_\alpha(\varepsilon))$ for all $\varepsilon \in q^{-1}(U_\alpha)$.
\end{enumerate}
\end{lemma}

\begin{proof}
For part~\cref{item: z-action is a homeo of EE}, fix $z \in \T$. Since the action of $\T$ on $\EE$ is continuous, $\varphi_z$ is a continuous bijection with inverse $\varphi_{\overline{z}}$, and hence $\varphi_z$ is a homeomorphism. Part~\cref{item: q and action} is \cite[Lemma~11.1.3]{Sims2020}. For part~\cref{item: t_alpha cts}, fix $\alpha \in \GG$ and $\varepsilon \in q^{-1}(U_\alpha)$. Since $\phi_{P_\alpha}\colon U_\alpha \times \T \to q^{-1}(U_\alpha)$ is a homeomorphism, there is a unique pair $(\beta_\varepsilon, z_\varepsilon) \in U_\alpha \times \T$ such that $\varepsilon = \phi_{P_\alpha}(\beta_\varepsilon, z_\varepsilon) = z_\varepsilon \cdot P_\alpha(\beta_\varepsilon)$. Since $q \circ P_\alpha = \id_{U_\alpha}$, we have $q(\varepsilon) = \beta_\varepsilon$, and since $\phi_{P_\alpha}$ is a homeomorphism, there is a unique continuous map $t_\alpha\colon q^{-1}(U_\alpha) \to \T$ given by $t_\alpha(\varepsilon) \coloneqq z_\varepsilon$.
\end{proof}

We now show that the continuous local sections of \cref{def: twist}\labelcref{cond: cts local section} can always be chosen to be defined on bisections of $\GG$, and to map units of $\GG$ to units of $\EE$.

\begin{lemma} \label{lemma: CLS bisections and units}
Every twist $(\EE,i,q)$ over a Hausdorff \'etale groupoid $\GG$ has a local trivialisation $(B_\alpha, P_\alpha, \phi_{P_\alpha})_{\alpha \in \GG}$ such that for all $\alpha \in \GG$, $B_\alpha$ is a bisection and $P_\alpha(B_\alpha \cap \GGo) \subseteq \EEo$.
\end{lemma}

\begin{proof}
Let $(\EE,i,q)$ be a twist over a Hausdorff \'etale groupoid $\GG$, and let $(U_\alpha, S_\alpha, \psi_{S_\alpha})_{\alpha \in \GG}$ be a local trivialisation of $(\EE,i,q)$. For each $\alpha \in \GG$, let $D_\alpha$ be an open bisection of $\GG$ such that $\alpha \in D_\alpha \subseteq U_\alpha$, and define
\[
B_\alpha \coloneqq \begin{cases}
D_\alpha \cap \GGo & \text{if } \alpha \in \GGo \\
D_\alpha {\setminus} \GGo & \text{if } \alpha \notin \GGo.
\end{cases}
\]
Since $\GG$ is a Hausdorff \'etale groupoid, $\GGo$ is clopen, and hence each $B_\alpha$ is an open bisection of $\GG$ containing $\alpha$.

There are now two cases to deal with. For the first case, fix $\alpha \in \GG {\setminus} \GGo$. Define $P_\alpha \coloneqq S_\alpha\restr{B_\alpha}$ and $\phi_{P_\alpha} \coloneqq \psi_{S_\alpha}\restr{B_\alpha \times \T}$. It is clear that $P_\alpha$ is a continuous local section for $q$ satisfying $P_\alpha(B_\alpha \cap \GGo) \subseteq \EEo$, and that $\phi_{P_\alpha}$ satisfies \cref{def: twist}\labelcref{cond: CLS induces homeo}. For the second case, fix $x \in \GGo$. Then $B_x \subseteq \GGo$. Define $P_x\colon B_x \to \EE$ by $P_x(y) \coloneqq (q\restr{\EEo})^{-1}(y)$. Then $P_x$ is continuous because $q\restr{\EEo}$ a homeomorphism and the inclusion map $\EEo \hookrightarrow \EE$ is continuous. It is clear that $q \circ P_x = \id_{B_x}$, and so $P_x$ is a continuous local section for $q$ satisfying $P_x(B_x \cap \GGo) \subseteq \EEo$. By \cref{lemma: T-action properties}\cref{item: t_alpha cts}, there is a unique continuous map $t_x\colon q^{-1}(U_x) \to \T$ such that $\psi_{S_x}^{-1}(\varepsilon) = (q(\varepsilon), t_x(\varepsilon))$ for all $\varepsilon \in q^{-1}(U_x)$. In particular, for all $y \in B_x$, we have
\begin{equation} \label{eqn: P_x(y) special case}
\psi_{S_x}^{-1}\big(P_x(y)\big) = \big(q(P_x(y)), \, t_x(P_x(y))\big) = \big(y, t_x(P_x(y))\big).
\end{equation}
Define $f_x\colon B_x \times \T \to B_x \times \T$ by $f_x(y,z) \coloneqq \big(y, z \, t_x(P_x(y))\big)$. Since $t_x$ and $P_x$ are continuous, $f_x$ is a homeomorphism with inverse $f_x^{-1}\colon (y,z) \mapsto \big(y, z \, \overline{t_x(P_x(y))}\big)$. Define $\phi_{P_x} \coloneqq \psi_{S_x} \circ f_x$. Then $\phi_{P_x}$ is a homeomorphism from $B_x \times \T$ onto $q^{-1}(B_x)$. Fix $(y,z) \in B_x \times \T$. Using the definition of $\psi_{S_x}$ and that $i$ is a homomorphism for the second equality and using \cref{eqn: P_x(y) special case} for the final equality, we see that
\begin{align*}
\phi_{P_x}(y,z) = \psi_{S_x}\!\big(y, z \, t_x(P_x(y))\big) &= i(y,z) \, i\big(y, t_x(P_x(y))\big) \, S_x(y) \\
&= i(y,z) \, \psi_{S_x}\!\big(y, t_x(P_x(y))\big) = i(y,z) \, P_x(y).
\end{align*}
Thus we have constructed a local trivialisation $(B_\alpha, P_\alpha, \phi_{P_\alpha})_{\alpha \in \GG}$ for $(\EE,i,q)$ with the desired properties.
\end{proof}

\begin{remark}
In some texts (see, for instance, \cite[Definition~3.1]{Boenicke2021}), the existence of continuous local sections (and of the induced local trivialisation) is omitted from the definition of a twist $(\EE,i,q)$, and instead, $i$ is defined to be a homeomorphism onto the open set $q^{-1}(\GGo)$, and $q$ is defined to be a continuous open surjection. These conditions imply that $q$ admits continuous local sections (see \cite[Proposition~3.4]{Boenicke2021}), and since $i$ has a continuous inverse defined on $q^{-1}(\GGo)$, an argument similar to the one used in the proof of \cite[Proposition~4.8(c)]{ACCLMR2022} shows that these local sections induce a local trivialisation of the twist. Hence the twists of \cite[Definition~3.1]{Boenicke2021} are twists in the sense of \cref{def: twist}. On the other hand, in \cref{lemma: properties of twist maps} we show that, given a twist $(\EE,i,q)$ in the sense of \cref{def: twist}, the map $i$ is a homeomorphism onto the open set $q^{-1}(\GGo)$, and the map $q$ is a continuous open surjection. Hence \cref{def: twist} is in fact equivalent to \cite[Definition~3.1]{Boenicke2021}.
\end{remark}

\begin{lemma} \label{lemma: properties of twist maps}
Let $(\EE,i,q)$ be a twist over a Hausdorff \'etale groupoid $\GG$.
\begin{enumerate}[label=(\alph*)]
\item \label{item: q is an open quotient map} The map $q$ is a continuous open surjection, and $\GG$ has the quotient topology.
\item \label{item: open twist multiplication} The range, source, and multiplication maps on $\EE$ are all open.
\item \label{item: i is a homeo onto open set} The map $i$ is a homeomorphism onto $q^{-1}(\GGo)$, which is an open subset of $\EE$.
\end{enumerate}
\end{lemma}

\begin{proof}
For part~\cref{item: q is an open quotient map}, note that $q$ is a continuous surjection by \cref{def: twist}. We first show that $\GG$ has the quotient topology. Let $X$ be a subset of $\GG$. If $X$ is open, then $q^{-1}(X)$ is open in $\EE$, because $q$ is continuous. Suppose instead that $q^{-1}(X)$ is open in $\EE$. We must show that $X$ is open in $\GG$. Choose a local trivialisation $(U_\alpha, P_\alpha, \phi_{P_\alpha})_{\alpha \in \GG}$ of $(\EE,i,q)$. Fix $\alpha \in X$. The set $q^{-1}(X) \cap q^{-1}(U_\alpha)$ is an open subset of $q^{-1}(U_\alpha)$ that is closed under the action of $\T$ on $\EE$, and hence its (open) image under $\phi_{P_\alpha}^{-1}$ is of the form $V_\alpha \times \T$, for some open subset $V_\alpha$ of $U_\alpha$. We have
\[
V_\alpha = q\big(\phi_{P_\alpha}(V_\alpha \times \T)\big) = q\big(q^{-1}(X) \cap q^{-1}(U_\alpha)\big) = X \cap U_\alpha \subseteq X,
\]
and so $V_\alpha$ is an open neighbourhood of $\alpha$ contained in $X$. Hence $X$ is open in $\GG$, and $q$ is a quotient map. We now show that $q$ is an open map. Let $Y$ be an open subset of $\EE$. Since $\GG$ has the quotient topology, $q(Y)$ is open in $\GG$ if and only if $q^{-1}(q(Y))$ is open in $\EE$. Recall from \cref{lemma: T-action properties}\cref{item: z-action is a homeo of EE} that for each $z \in \T$, the map $\varphi_z\colon \varepsilon \mapsto z \cdot \varepsilon$ is a homeomorphism of $\EE$, and so $\varphi_z(Y)$ is open. Since $\varepsilon \in q^{-1}(q(Y))$ if and only if $\varepsilon = z \cdot \zeta$ for some $z \in \T$ and $\zeta \in Y$, we have $q^{-1}(q(Y)) = \T \cdot Y = \medcup_{z \in \T} \, \varphi_z(Y)$, which is open.

For part~\cref{item: open twist multiplication}, note that the range and source maps of $\GG$ are open because $\GG$ is \'etale. Since $q$ restricts to a homeomorphism of unit spaces, it follows that the range and source maps of $\EE$ are open, and thus the multiplication map on $\EE$ is open by \cite[Lemma~8.4.11]{Sims2020}.

For part~\cref{item: i is a homeo onto open set}, note that $i(\GGo \times \T) = q^{-1}(\GGo)$ is open in $\EE$ because $\GGo$ is open in $\GG$. To see that $i$ is an open map, let $U \subseteq \GGo$ and $W \subseteq \T$ be open sets, and use \cref{lemma: CLS bisections and units} to find a local trivialisation $(B_\alpha, P_\alpha, \phi_{P_\alpha})_{\alpha \in \GG}$ of $(\EE,i,q)$ such that for all $\alpha \in \GG$, $B_\alpha$ is a bisection of $\GG$ and $P_\alpha(B_\alpha \cap \GGo) \subseteq \EEo$. For each $x \in U$, define $D_x \coloneqq B_x \cap U$, so that $D_x \subseteq \GGo$ and $U = \bigcup_{x \in U} D_x$. Fix $x \in U$. Since $P_x(D_x) \subseteq \EEo$, we have $\phi_{P_x}(D_x \times W) = i(D_x \times W)$. By \cref{def: twist}\labelcref{cond: CLS induces homeo}, $\phi_{P_x}$ is a homeomorphism onto the open set $q^{-1}(B_x)$, and so since $D_x \times W$ is open, it follows that $i(D_x \times W)$ is open in $\EE$. Hence $i(U \times W) = \bigcup_{x \in U} i(D_x \times W)$ is an open subset of $\EE$, and thus $i$ is an open map.
\end{proof}

\begin{definition}
Let $(\EE,i,q)$ be a twist over a Hausdorff \'etale groupoid $\GG$. Any map $P\colon \GG \to \EE$ satisfying $q \circ P = \id_{\GG}$ is called a \hl{(global) section} for $q$.
\end{definition}

The following result shows that there is a one-to-one correspondence between continuous $2$-cocycles on a Hausdorff \'etale groupoid $\GG$ and twists over $\GG$ admitting \emph{continuous} global sections. Note that in general, a twist over a Hausdorff \'etale groupoid need not admit any continuous global sections.

\begin{prop} \label{prop: twists and 2-cocycles}
Let $(\EE,i,q)$ be a twist over a Hausdorff \'etale groupoid $\GG$. Suppose that $P\colon \GG \to \EE$ is a continuous global section for $q$. Then there is a continuous $2$-cocycle $\sigma\colon \GGc \to \T$ such that $P(\alpha) \, P(\beta) = \sigma(\alpha,\beta) \cdot P(\alpha\beta)$ for all $(\alpha,\beta) \in \GGc$. Let $(\EE_\sigma, i_\sigma, q_\sigma)$ be the twist defined in \cref{eg: EE_sigma}. The map $\phi_P\colon \EE_\sigma \to \EE$ given by $\phi_P(\gamma,z) \coloneqq z \cdot P(\gamma)$ defines an isomorphism of twists, in the sense that $\phi_P$ is a topological groupoid isomorphism that makes the diagram
\[
\begin{tikzcd}
{} & {\EE_\sigma} \arrow{dd}[swap]{\phi_P} \arrow{dr}{q_\sigma} & {} \\[-2ex]
{\GGo \times \T} \arrow{ur}{i_\sigma} \arrow{dr}{i} && \GG \\[-2ex]
& \EE \arrow{ur}{q} & {}
\end{tikzcd}
\]
commute. Moreover, there is a continuous global section $S\colon \GG \to \EE$ for $q$ satisfying $S(\GGo) \subseteq \EEo$.
\end{prop}

\begin{proof}
By \cite[Section~4, Fact~1]{Kumjian1986}, every continuous global section for $q$ induces a continuous $2$-cocycle satisfying the given formula. It is observed in \cite[Section~4, Remark~2]{Kumjian1986} that the map $\phi_P\colon (\gamma,z) \mapsto z \cdot P(\gamma)$ defines an isomorphism of the twists $(\EE_\sigma, i_\sigma, q_\sigma)$ and $(\EE, i, q)$. (Alternatively, see the proof of the analogous result \cite[Proposition~4.8]{ACCLMR2022} for discrete twists, which holds in our non-discrete setting.) To see that continuous global sections can be chosen to map units to units, observe that $(\GG, P, \phi_P)_{\alpha \in \GG}$ is a local trivialisation of $\EE$, and so we can apply the argument of \cref{lemma: CLS bisections and units} (without replacing $\GG$ by bisections in the local trivialisation).
\end{proof}

The following result and the subsequent corollary will be frequently used throughout the remainder of the article, without necessarily being explicitly referenced. These results allow us to restrict our attention to twists over the unit space, the interior of the isotropy, or the isotropy groups of a Hausdorff \'etale groupoid.

\begin{lemma} \label{lemma: subgroupoid twists}
Let $(\EE,i,q)$ be a twist over a Hausdorff \'etale groupoid $\GG$. Suppose that $\HH$ is an open or closed \'etale subgroupoid of $\GG$. Define $\EE_\HH \coloneqq q^{-1}(\HH)$, $i_\HH \coloneqq i\restr{\HHo \times \T}$, and $q_\HH \coloneqq q\restr{\EE_\HH}$. Then $\HH$ is a Hausdorff \'etale groupoid, and $(\EE_\HH, i_\HH, q_\HH)$ is a twist over $\HH$.
\end{lemma}

\begin{proof}
The argument used in the proof of \cite[Lemma~2.7]{BFPR2021} applies in both the open and closed settings.
\end{proof}

\begin{cor} \label{cor: specific subgroupoid twists}
Let $(\EE,i,q)$ be a twist over a Hausdorff \'etale groupoid $\GG$.
\begin{enumerate}[label=(\alph*)]
\item \label{item: unit space twist} The groupoid $q^{-1}(\GGo)$ is a twist over $\GGo$.
\item \label{item: isotropy subgroupoid twist} The isotropy subgroupoid $\IE$ of $\EE$ is a twist over the isotropy subgroupoid $\IG$ of $\GG$.
\item \label{item: isotropy group twist} For each $u \in \EEo$, the isotropy group $\IE_u$ is a twist over the isotropy group $\IG_{q(u)}$.
\end{enumerate}
\end{cor}

\begin{proof}
Part~\cref{item: unit space twist} follows immediately from \cref{lemma: subgroupoid twists} since $\GGo$ is open.

For part~\cref{item: isotropy subgroupoid twist}, we will show that $q^{-1}(\IG) = \IE$, because the result then follows immediately from \cref{lemma: subgroupoid twists} since $\IG$ is an open \'etale subgroupoid of $\GG$. We first show that $q^{-1}\big(\!\Iso(\GG)\big) = \Iso(\EE)$. For this, fix $\varepsilon \in \EE$. Since $q\colon \EE \to \GG$ is a groupoid homomorphism that restricts to a homeomorphism of unit spaces, we have
\[
r(q(\varepsilon)) = s(q(\varepsilon)) \iff q(r(\varepsilon)) = q(s(\varepsilon)) \iff r(\varepsilon) = s(\varepsilon),
\]
and so $q(\varepsilon) \in \Iso(\GG)$ if and only if $\varepsilon \in \Iso(\EE)$. Thus $q^{-1}\big(\!\Iso(\GG)\big) = \Iso(\EE)$, as claimed. Since $\IG$ is open and $q$ is continuous, $q^{-1}(\IG)$ is an open subset of $\EE$ contained in $\Iso(\EE)$, and so $q^{-1}(\IG) \subseteq \IE$. Since $\IE$ is open and $q$ is an open map by \cref{lemma: properties of twist maps}\cref{item: q is an open quotient map}, $q(\IE)$ is an open subset of $\GG$ contained in $\Iso(\GG)$, and hence $\IE \subseteq q^{-1}(\IG)$. Therefore, $q^{-1}(\IG) = \IE$.

For part~\cref{item: isotropy group twist}, fix $u \in \EEo$. The proof of part~\cref{item: isotropy subgroupoid twist} implies that $q^{-1}\big(\IG_{q(u)}\big) = \IE_u$, and so the result follows from \cref{lemma: subgroupoid twists} since $\IG_{q(u)}$ is a discrete closed subgroupoid of $\IG$.
\end{proof}

\subsection{Twisted groupoid \texorpdfstring{C*}{C*}-algebras}
\label{subsec: twisted groupoid C*s}

In this section we recall Kumjian's construction (given in \cite{Kumjian1986}) of the full and reduced twisted groupoid C*-algebras associated to a twist over a Hausdorff \'etale groupoid. We also recall Renault's construction (given in \cite{Renault1980}) of the full and reduced twisted groupoid C*-algebras arising from a continuous $2$-cocycle on a Hausdorff \'etale groupoid, and we describe the relationship between these two constructions in \cref{prop: C*s of twists and 2-cocycles}.

Given a locally compact Hausdorff space $X$, we write $C(X)$ for the vector space of continuous complex-valued functions on $X$ under pointwise operations. For each $f \in C(X)$, we define $\osupp(f) \coloneqq f^{-1}(\C {\setminus} \{0\})$, and we write $\supp(f)$ for the closure of $\supp(f)$ in $X$. We define $C_c(X) \coloneqq \{ f \in C(X) : \supp(f) \text{ is compact} \}$, and we write $C_0(X)$ for the collection of continuous functions on $X$ vanishing at infinity, which is the completion of the subspace $C_c(X)$ with respect to the uniform norm $\lv \cdot \rv_\infty$.

Suppose that $(\EE,i,q)$ is a twist over a Hausdorff \'etale groupoid $\GG$. For each $\gamma \in \GG$, the set $q^{-1}(\gamma)$ is homeomorphic to $\T$, since $\EE$ is a locally trivial $\GG$-bundle. Since the Haar measure on $\T$ is rotation-invariant, pulling it back to $q^{-1}(\gamma)$ gives a measure that is independent of our choice of $\varepsilon \in q^{-1}(\gamma)$. For each $u \in \EEo$, we endow $\EE_u$ with a measure $\lambda_u$ that agrees with these pulled back copies of Haar measure on $q^{-1}(\gamma)$ for each $\gamma \in \GG_{q(u)}$, and so each $q^{-1}(\gamma)$ has measure $1$. We define a measure $\lambda^u$ on each $\EE^u$ in a similar fashion. We say that a function $f\colon \EE \to \C$ is \hl{$\T$-equivariant} if $f(z \cdot \varepsilon) = z \, f(\varepsilon)$ for all $z \in \T$ and $\varepsilon \in \EE$. The collection
\[
\Sigma_c(\GG;\EE) \coloneqq \left\{ f \in C_c(\EE) : f \text{ is $\T$-equivariant} \right\}
\]
is a $*$-algebra under pointwise addition and scalar multiplication, multiplication given by the convolution formula
\[
(fg)(\varepsilon) \coloneqq \int_{\EE_{s(\varepsilon)}} f(\varepsilon\zeta^{-1}) \, g(\zeta) \, \d \lambda_{s(\varepsilon)}(\zeta) = \int_{\EE^{r(\varepsilon)}} f(\eta) \, g(\eta^{-1}\varepsilon) \, \d \lambda^{r(\varepsilon)}(\eta),
\]
and involution given by $f^*(\varepsilon) \coloneqq \overline{f(\varepsilon^{-1})}$. Taking $P\colon \GG \to \EE$ to be any (not necessarily continuous) global section for $q$, it follows from the $\T$-equivariance of $f, g \in \Sigma_c(\GG;\EE)$ that for all $\varepsilon \in \EE$,
\begin{equation} \label{eqn: convolution as a finite sum}
(fg)(\varepsilon) = \sum_{\alpha \in \GG_{q(s(\varepsilon))}} f\big(\varepsilon P(\alpha)^{-1}\big) \, g\big(P(\alpha)\big) = \sum_{\beta \in \GG^{q(r(\varepsilon))}} f\big(P(\beta)\big) \, g\big(P(\beta)^{-1} \varepsilon\big).
\end{equation}

\begin{remark}
In \cite[Section~2]{Kumjian1986} Kumjian observes that $\Sigma_c(\GG;\EE)$ can alternatively be regarded as a collection of sections of the complex line bundle $(\C \times \EE)/\T$ over $\GG$, but we will only make use of this description when referencing external results that use it.
\end{remark}

Although $C_c(\GG)$ is spanned by functions supported on open bisections of $\GG$ (see, for instance, \cite[Lemma~9.1.3]{Sims2020}), this is not the case for $\Sigma_c(\GG;\EE)$, because $\EE$ is not \'etale. However, we do have the following similar result.

\begin{lemma} \label{lemma: bisection span}
Let $(\EE,i,q)$ be a twist over a Hausdorff \'etale groupoid $\GG$. Then
\[
\Sigma_c(\GG;\EE) = \vecspan\!\left\{ f \in \Sigma_c(\GG;\EE) : q(\supp(f)) \text{ is a bisection of } \GG \right\}\!.
\]
\end{lemma}

\begin{proof}
Fix $f \in \Sigma_c(\GG;\EE)$. Since $q(\supp(f))$ is compact, it can be covered by a finite collection $\{ B_\gamma : \gamma \in F \}$ of open bisections of $\GG$. As in \cite[Remark~2.9]{LPS2014}, let $\{h_\gamma : \gamma \in F\}$ be a partition of unity subordinate to $\{B_\gamma \cap q(\supp(f)) : \gamma \in F\}$. For each $\gamma \in F$, the function $f_\gamma\colon \varepsilon \mapsto f(\varepsilon) \, h_\gamma(q(\varepsilon))$ belongs to $\Sigma_c(\GG;\EE)$, and $q(\supp(f_\gamma)) \subseteq \supp(h_\gamma) \subseteq B_\gamma$. Fix $\gamma \in F$. Since $\sum_{\gamma \in F} h_\gamma = 1$, we have $f = \sum_{\gamma \in F} f_\gamma$, which completes the proof.
\end{proof}

The \hl{full twisted groupoid C*-algebra} associated to the pair $(\GG,\EE)$ is defined to be the completion $C^*(\GG;\EE)$ of $\Sigma_c(\GG;\EE)$ with respect to the \hl{full norm}
\[
\lv f \rv \coloneqq \sup\!\left\{ \lv \pi(f) \rv : \pi \text{ is a $*$-representation of } \Sigma_c(\GG;\EE) \right\}.
\]

For each unit $u \in \EEo$, there is a $*$-representation $\pi_u$ of $\Sigma_c(\GG;\EE)$ on the Hilbert space $L^2(\GG_{q(u)};\EE_u)$ consisting of square-integrable $\T$-equivariant functions on $\EE_u$, which is given by extension of the convolution formula. We call each $\pi_u$ the \hl{regular representation} of $\Sigma_c(\GG;\EE)$ associated to $u$, and we write $\pi_u^\II$ for the regular representation of $\Sigma_c(\IG;\IE)$ associated to $u$. The \hl{reduced twisted groupoid C*-algebra} $C_r^*(\GG;\EE)$ is defined to be the completion of $\Sigma_c(\GG;\EE)$ with respect to the \hl{reduced norm}
\[
\lv f \rv_r \coloneqq \sup\!\left\{ \lv \pi_u(f) \rv : u \in \EEo \right\}.
\]
For all $f \in \Sigma_c(\GG;\EE)$, we have $\lv f \rv_\infty \le \lv f \rv_r \le \lv f \rv$, with equality throughout if $q(\supp(f))$ is a bisection of $\GG$. If $\GG$ is amenable, then the full and reduced norms agree on $\Sigma_c(\GG;\EE)$.

We define
\[
D_0 \coloneqq \left\{ f \in \Sigma_c(\GG;\EE) : q(\supp(f)) \subseteq \GGo \right\} = \left\{ f \in \Sigma_c(\GG;\EE) : \supp(f) \subseteq i(\GGo \times \T) \right\}\!,
\]
and note that there is a $*$-isomorphism $C_c(\GGo) \cong D_0$ mapping $h \in C_c(\GGo)$ to the function $f_h\colon i(x,z) \mapsto z \, h(x)$ (see \cite[Lemma~11.1.9]{Sims2020}). This $*$-isomorphism extends to an isomorphism of $C_0(\GGo)$ to the completion $D_r$ of $D_0$ in $C_r^*(\GG;\EE)$. There is a faithful conditional expectation $\Phi_r\colon C_r^*(\GG;\EE) \to D_r$ that extends restriction of functions in $\Sigma_c(\GG;\EE)$ to the set $q^{-1}(\GGo) = i(\GGo \times \T)$ (see \cite[Proposition~4.3]{Renault2008} and \cite[Proposition~11.1.13]{Sims2020}). We write $\Phi_r^\II$ for the corresponding conditional expectation from $C_r^*(\IG;\IE)$ to $D_r$. There is also a conditional expectation from $C^*(\GG;\EE)$ to the completion of $D_0$ in the full norm that extends restriction of functions from $\Sigma_c(\GG;\EE)$ to $q^{-1}(\GGo)$, but this conditional expectation is not necessarily faithful.

We now show that every $*$-homomorphism of $\Sigma_c(\GG;\EE)$ into a C*-algebra $A$ extends uniquely to a C*-homomorphism of $C^*(\GG;\EE)$ into $A$. This is an extension of the well known result \cite[Lemma~3.3.22]{Armstrong2019} to the setting of C*-algebras of groupoid twists.

\begin{lemma} \label{lemma: *-hom extension to full C*}
Let $(\EE,i,q)$ be a twist over a Hausdorff \'etale groupoid $\GG$. Suppose that $A$ is a C*-algebra and that $\phi\colon \Sigma_c(\GG;\EE) \to A$ is a $*$-homomorphism. Then $\phi$ extends uniquely to a C*-homomorphism $\overline{\phi}\colon C^*(\GG;\EE) \to A$ satisfying $\lv \overline{\phi} \rv = \lv \phi \rv$.
\end{lemma}

\begin{proof}
The result follows by a similar argument to the proof of \cite[Lemma~3.3.22]{Armstrong2019}.
\end{proof}

We now recall Renault's construction of twisted groupoid C*-algebras arising from continuous groupoid $2$-cocycles. For our purposes, it suffices to consider a Hausdorff \'etale groupoid $\GG$, although we note that Renault deals with groupoids that are not necessarily \'etale in \cite{Renault1980}. Suppose that $\sigma\colon \GGc \to \T$ is a continuous $2$-cocycle. We write $C_c(\GG,\sigma)$ for the $*$-algebra consisting of continuous, compactly supported, complex-valued functions on $\GG$ equipped with pointwise addition and scalar multiplication, multiplication given by the twisted convolution formula
\[
(fg)(\gamma) \coloneqq \sum_{\substack{(\alpha,\beta) \in \GGc, \\ \alpha\beta = \gamma}} \sigma(\alpha,\beta) \, f(\alpha) \, g(\beta) = \sum_{\eta \in \GG_{s(\gamma)}} \sigma(\gamma\eta^{-1},\eta) \, f(\gamma\eta^{-1}) \, g(\eta),
\]
and involution given by $f^*(\gamma) \coloneqq \overline{\sigma(\gamma,\gamma^{-1})} \, \overline{f(\gamma^{-1})}$. The full and reduced norms on $C_c(\GG,\sigma)$ are defined in a similar fashion to the full and reduced norms on $\Sigma_c(\GG;\EE)$, and by completing $C_c(\GG,\sigma)$ with respect to these norms, we obtain the full and reduced twisted groupoid C*-algebras $C^*(\GG,\sigma)$ and $C_r^*(\GG,\sigma)$, respectively.

In \cref{subsec: twists} we showed that every continuous $2$-cocycle $\sigma$ on $\GG$ gives rise to a twist $\EE_\sigma$ over $\GG$ by $\T$, and in \cref{prop: twists and 2-cocycles} we showed that every twist $\EE$ admitting a continuous global section gives rise to a continuous $2$-cocycle $\sigma$ such that $\EE_\sigma \cong \EE$. If $\phi\colon \EE_1 \to \EE_2$ is an isomorphism of twists over $\GG$, then a routine argument shows that the map $f \mapsto f \circ \phi$ is an isomorphism from $\Sigma_c(\GG;\EE_2)$ to $\Sigma_c(\GG;\EE_1)$. Thus we can use \cref{prop: twists and 2-cocycles} to describe the relationship between twisted groupoid C*-algebras arising from continuous $2$-cocycles, and those arising from twists admitting continuous global sections, as follows.

\begin{prop} \label{prop: C*s of twists and 2-cocycles}
Let $(\EE,i,q)$ be a twist over a Hausdorff \'etale groupoid $\GG$. Suppose that $P\colon \GG \to \EE$ is a continuous global section for $q$. Let $\sigma\colon \GGc \to \T$ be the continuous $2$-cocycle induced by $P$, as described in \cref{prop: twists and 2-cocycles}. Let $\overline{\sigma}$ denote the continuous $2$-cocycle obtained by composing $\sigma$ with the complex conjugation map on $\T$. There is a $*$-isomorphism $\psi_P\colon \Sigma_c(\GG;\Sigma) \to C_c(\GG,\overline{\sigma})$ given by $\psi_P(f) \coloneqq f \circ P$. This isomorphism extends to C*-isomorphisms $C^*(\GG;\EE) \cong C^*(\GG,\overline{\sigma})$ and $C_r^*(\GG;\EE) \cong C_r^*(\GG,\overline{\sigma})$.
\end{prop}

\begin{proof}
Let $(\EE_\sigma, i_\sigma, q_\sigma)$ be the twist defined in \cref{eg: EE_sigma}. Since $\EE \cong \EE_\sigma$ by \cref{prop: twists and 2-cocycles}, and since $(\gamma,1) \mapsto \gamma$ is a continuous global section for $q_\sigma$, the result follows from the ``$n = -1$'' case of \cite[Lemma~3.1(b)]{BaH2014}.
\end{proof}

\begin{remark} \label{rem: translations}
Using \cref{prop: C*s of twists and 2-cocycles}, the results in \cref{sec: twisted isotropy C*s,sec: uniqueness} can be translated into analogous results for twisted groupoid C*-algebras arising from continuous $2$-cocycles. Alternatively, we refer the reader to \cite[Chapter~5]{Armstrong2019} for the corresponding results written explicitly in this framework.
\end{remark}

\section{A universal property for twisted group \texorpdfstring{C*}{C*}-algebras}
\label{sec: twisted group C*s}

In this section, we describe twisted C*-algebras of countable discrete groups. Since every such group is a Hausdorff \'etale groupoid, the preliminaries given in \cref{sec: prelim} are all applicable here. In particular, since every twist over a discrete group admits a (trivially continuous) global section, \cref{prop: twists and 2-cocycles} applies to twists over countable discrete groups. Thus, a twisted group C*-algebra can be viewed as arising either from a continuous $\T$-valued $2$-cocycle on the group, or from a central extension of the group by $\T$. The main purpose of this section is to translate the universal property for the full twisted C*-algebra $C^*(G,\sigma)$ associated to a countable discrete group $G$ and a $\T$-valued $2$-cocycle $\sigma$ on $G$ into the language of the associated central extension $G \times_\sigma \T$ of $G$ by $\T$ (see \cref{thm: universal property}). This result is presumably well known, but we were unable to find an explicit statement of it in the literature.

Suppose that $G$ is a countable discrete group with identity $e$, and that $\sigma\colon G \times G \to \T$ is a $2$-cocycle. Given a C*-algebra $A$ with identity $1_A$, we say that $u\colon G \to A$ is a \hl{$\sigma$-twisted unitary representation} of $G$ in $A$ if each $u_g$ is a unitary element of $A$, and $u_g u_h = \sigma(g,h) \, u_{gh}$ for all $g, h \in G$. This implies that $u_e = 1_A$, and $u_g^* = \overline{\sigma(g,g^{-1})} \, u_{g^{-1}}$ for each $g \in G$. The map $\delta\colon G \to C^*(G,\sigma)$ sending each $g \in G$ to the point-mass function $\delta_g$ at $g$ is a $\sigma$-twisted unitary representation of $G$ such that $C^*(G,\sigma) = C^*\big(\{\delta_g : g \in G \}\big)$, and the following universal property holds: if $B$ is a unital C*-algebra and $u\colon G \to B$ is a $\sigma$-twisted unitary representation of $G$ in $B$, then there is a homomorphism $\lambda_u\colon C^*(G,\sigma) \to B$ such that $\lambda_u(\delta_g) = u_g$ for each $g \in G$.

\begin{prop} \label{prop: delta^T_varepsilon}
Let $(E,i,q)$ be a twist over a countable discrete group $G$. Fix $\varepsilon \in E$, and define $\delta^\T_\varepsilon\colon E \to \C$ by
\[
\delta^\T_\varepsilon(\zeta) \coloneqq \begin{cases}
w \ & \text{if } \zeta = w \cdot \varepsilon \text{ for some } w \in \T \\
0 \ & \text{if } \zeta \ne w \cdot \varepsilon \text{ for all } w \in \T.
\end{cases}
\]
Then $\Sigma_c(G;E) = \vecspan\!\left\{ \delta^\T_\varepsilon : \varepsilon \in E \right\}$, and $C_r^*(G;E)$ and $C^*(G;E)$ are both unital with identity $\delta^\T_e$, where $e$ is the identity of $E$. For all $\varepsilon, \zeta \in E$ and $z \in \T$, $\delta^\T_\varepsilon$ is a unitary element of $\Sigma_c(G;E)$ satisfying $(\delta^\T_\varepsilon)^* = \delta^\T_{\varepsilon^{-1}}$, and we have $\delta^\T_{z \cdot \varepsilon} = \overline{z} \, \delta^\T_\varepsilon$ and $\delta^\T_\varepsilon \delta^\T_\zeta = \delta^\T_{\varepsilon\zeta}$.
\end{prop}

\begin{proof}
By regarding $\Sigma_c(G;E)$ as sections of a line bundle (as in \cite[Section~2]{BFPR2021}) and by performing some routine calculations, it can be seen that the material following \cite[Remark~2.5]{BFPR2021} implies the result.
\end{proof}

In the following theorem, we use the universal property of $C^*(G,\overline{\sigma})$ to give a universal property for $C^*(G;E)$.

\begin{thm} \label{thm: universal property}
Let $(E,i,q)$ be a twist over a countable discrete group $G$. Suppose that $\{ t_\varepsilon : \varepsilon \in E \}$ is a family of unitary elements of a unital C*-algebra $A$ such that $t_{z \cdot \varepsilon} = \overline{z} \, t_\varepsilon$ and $t_\varepsilon \, t_\zeta = t_{\varepsilon\zeta}$ for all $\varepsilon, \zeta \in E$ and $z \in \T$. Then there is a homomorphism $\pi_t\colon C^*(G;E) \to A$ satisfying $\pi_t(\delta^\T_\varepsilon) = t_\varepsilon$ for each $\varepsilon \in E$, where $\delta^\T_\varepsilon$ is defined as in \cref{prop: delta^T_varepsilon}.
\end{thm}

\begin{proof}
Recall from \cref{prop: twists and 2-cocycles,prop: C*s of twists and 2-cocycles} that there is a $2$-cocycle $\sigma\colon G \times G \to \T$ such that $E \cong G \times_\sigma \T$ and $C^*(G;E) \cong C^*(G,\overline{\sigma})$. Let $e$ be the identity of $G$, and let $1_A$ be the identity of $A$. By \cite[Example~2.8.14]{Echterhoff2017}, there is a one-to-one correspondence between $\sigma$-twisted unitary representations of $G$ in $A$ and unitary representations $t$ of $G \times_\sigma \T$ in $A$ that satisfy $t_{(e,z)} = \overline{z} \, 1_A$ for all $z \in \T$. Such representations satisfy $t_{z \cdot \varepsilon} = \overline{z} \, t_\varepsilon$ and $t_\varepsilon \, t_\zeta = t_{\varepsilon\zeta}$ for all $\varepsilon, \zeta \in E$. Thus the result follows from the universal property of $C^*(G,\overline{\sigma})$.
\end{proof}

\section{Twisted \texorpdfstring{C*}{C*}-algebras associated to the interior of the isotropy of a Hausdorff \'etale groupoid}
\label{sec: twisted isotropy C*s}

In this section we study the relationships between the full and reduced twisted C*-algebras associated to the interior $\IG$ of the isotropy of a Hausdorff \'etale groupoid $\GG$ and the full and reduced twisted C*-algebras associated to the isotropy groups of $\IG$. The results in this section are extensions of the results in \cite[Section~5.1]{Armstrong2019} to the setting of C*-algebras of groupoid twists.

We saw in \cref{cor: specific subgroupoid twists} that, given a twist $(\EE,i,q)$ over a Hausdorff \'etale groupoid $\GG$, the interior $\IE$ of the isotropy of $\EE$ is a twist over $\IG$, and for each $u \in \EEo$, the isotropy group $\IE_u \coloneqq \{ \varepsilon \in \IE : r(\varepsilon) = s(\varepsilon) = u \}$ is a twist over the countable discrete group $\IG_{q(u)}$. Thus, recalling the notation defined in \cref{prop: delta^T_varepsilon}, we have
\[
\Sigma_c\big(\IG_{q(u)};\IE_u\big) = \vecspan\{ \delta^\T_\varepsilon : \varepsilon \in \IE_u \}.
\]

\begin{thm} \label{thm: quotient maps}
Let $(\EE,i,q)$ be a twist over a Hausdorff \'etale groupoid $\GG$. Fix $u \in \EEo$, and define $F_u \coloneqq \left\{ f \in \Sigma_c(\IG;\IE) : f\restr{\IE_u} \equiv 0 \right\}$.
\begin{enumerate}[label=(\alph*)]
\item \label{item: reduced quotient map} Let $J_u^r$ denote the closure of $F_u$ in the reduced norm. Then $J_u^r$ is an ideal of $C_r^*(\IG;\IE)$, and there is a surjective $*$-homomorphism
\[
\theta_u^r\colon C_r^*(\IG;\IE) / J_u^r \to C_r^*(\IG_{q(u)}; \IE_u)
\]
satisfying $\theta_u^r(f + J_u^r) = f\restr{\IE_u}$ for all $f \in \Sigma_c(\IG;\IE)$.
\item \label{item: full quotient map} Let $J_u$ denote the closure of $F_u$ in the full norm. Then $J_u$ is an ideal of $C^*(\IG;\IE)$, and there is an isomorphism
\[
\theta_u\colon C^*(\IG;\IE) / J_u \to C^*(\IG_{q(u)}; \IE_u)
\]
satisfying $\theta_u(f + J_u) = f\restr{\IE_u}$ for all $f \in \Sigma_c(\IG;\IE)$.
\end{enumerate}
\end{thm}

\begin{remark}
Somewhat surprisingly, the map $\theta_u^r\colon C_r^*(\IG;\IE) / J_u^r \to C_r^*(\IG_{q(u)}; \IE_u)$ of \cref{thm: quotient maps}\cref{item: reduced quotient map} is not in general an isomorphism, unlike the analogue for the full C*-algebras in \cref{thm: quotient maps}\cref{item: full quotient map}. In \cref{thm: reduced quotient map not injective}, we prove this by using an example introduced by Willett in \cite{Willett2015} of a \emph{nonamenable} groupoid whose full and reduced C*-algebras coincide.
\end{remark}

In order to prove \cref{thm: quotient maps}, we need several preliminary results. We begin by showing that $J_u^r$ and $J_u$ are ideals of $C_r^*(\IG;\IE)$ and $C^*(\IG;\IE)$, respectively.

\begin{lemma} \label{lemma: ideals for quotients}
Let $(\EE,i,q)$ be a twist over a Hausdorff \'etale groupoid $\GG$. Fix $u \in \EEo$, and define $F_u \coloneqq \big\{ f \in \Sigma_c(\IG;\IE) : f\restr{\IE_u} \equiv 0 \big\}$. Then $F_u$ is an algebraic ideal of $\Sigma_c(\IG;\IE)$. Let $J_u^r$ and $J_u$ denote the closures of $F_u$ in the reduced and full norms, respectively. Then $J_u^r$ is an ideal of $C_r^*(\IG;\IE)$, and $J_u$ is an ideal of $C^*(\IG;\IE)$.
\end{lemma}

\begin{proof}
It is clear that $F_u$ is a linear subspace of $\Sigma_c(\IG;\IE)$. To see that $F_u$ is an algebraic ideal, fix $f \in F_u$ and $g \in \Sigma_c(\IG;\IE)$. For all $\varepsilon \in \IE_u$, we have $\varepsilon^{-1} \in \IE_u$, and so $f^*(\varepsilon) = \overline{f(\varepsilon^{-1})} = 0$. Thus $f^* \in F_u$. For all $\varepsilon \in \IE_u$, we have
\[
(fg)(\varepsilon) = \int_{\IE_u} f(\zeta) \, g(\zeta^{-1}\varepsilon) \, \d \lambda^u(\zeta) = 0 \quad \text{ and } \quad (gf)(\varepsilon) = \int_{\IE_u} g(\varepsilon\zeta^{-1}) \, f(\zeta) \, \d \lambda_u(\zeta) = 0.
\]
Hence $fg, gf \in F_u$, and thus $F_u$ is an algebraic ideal of $\Sigma_c(\IG;\IE)$. Since all C*-algebraic operations are continuous, it follows that $J_u^r$ and $J_u$ are ideals of $C_r^*(\IG;\IE)$ and $C^*(\IG;\IE)$, respectively.
\end{proof}

\begin{lemma} \label{lemma: j_varepsilon}
Let $(\EE,i,q)$ be a twist over a Hausdorff \'etale groupoid $\GG$. Fix $u \in \EEo$. For each $\varepsilon \in \IE_u$, define $\delta^\T_\varepsilon \in \Sigma_c\big(\IG_{q(u)};\IE_u\big)$ as in \cref{prop: delta^T_varepsilon}. Then for each $\varepsilon \in \IE_u$, there exists $j_\varepsilon \in \Sigma_c(\IG;\IE)$ such that $j_\varepsilon\restr{\IE_u} = \delta^\T_\varepsilon$ and $j_{w \cdot \varepsilon} = \overline{w} \, j_\varepsilon$ for each $w \in \T$.
\end{lemma}

\begin{proof}
By \cref{cor: specific subgroupoid twists}\cref{item: isotropy subgroupoid twist}, $\IE$ is a twist over $\IG$, and so by \cref{lemma: CLS bisections and units}, we can find a local trivialisation $(B_\alpha, P_\alpha, \phi_{P_\alpha})_{\alpha \in \IG}$ for $\IE$ such that for each $\alpha \in \IG$, $B_\alpha$ is an open bisection of $\IG$ containing $\alpha$. Fix $\varepsilon \in \IE_u$. Use Urysohn's lemma to choose $h_{q(\varepsilon)} \in C_c(\IG)$ such that $\supp(h_{q(\varepsilon)}) \subseteq B_{q(\varepsilon)}$ and $h_{q(\varepsilon)}(q(\varepsilon)) = 1$. Recall from \cref{lemma: T-action properties}\cref{item: t_alpha cts} that for each $\zeta \in q^{-1}(B_{q(\varepsilon)})$, there is a unique $z_\zeta \in \T$ such that $\zeta = \phi_{P_{q(\varepsilon)}}(q(\zeta),z_\zeta)$, and the map $\zeta \mapsto z_\zeta$ is continuous on $q^{-1}(B_{q(\varepsilon)})$. Thus $\zeta \mapsto z_\zeta \, h_{q(\varepsilon)}(q(\zeta))$ is a continuous map from $q^{-1}(B_{q(\varepsilon)})$ to $\C$. Define $j_\varepsilon\colon \IE \to \C$ by
\[
j_\varepsilon(\zeta) \coloneqq \begin{cases}
\overline{z_\varepsilon} \, z_\zeta \, h_{q(\varepsilon)}(q(\zeta)) & \ \text{if } \zeta \in q^{-1}(B_{q(\varepsilon)}) \\
0 & \ \text{if } \zeta \notin q^{-1}(B_{q(\varepsilon)}).
\end{cases}
\]
Then $\supp(j_\varepsilon) \subseteq q^{-1}\big(\!\supp(h_{q(\varepsilon)})\big) = \T \cdot P_{q(\varepsilon)}\big(\!\supp(h_{q(\varepsilon)})\big)$. Since $\T$ and $P_{q(\varepsilon)}\big(\!\supp(h_{q(\varepsilon)})\big)$ are compact sets and the action of $\T$ on $\IE$ is continuous, $\supp(j_\varepsilon)$ is compact. Since $j_\varepsilon\restr{q^{-1}(B_{q(\varepsilon)})}$ is continuous and $q^{-1}(B_{q(\varepsilon)})$ is open, $j_\varepsilon$ is continuous on all of $\IE$. To see that $j_\varepsilon$ is $\T$-equivariant, fix $\zeta \in \IE$ and $w \in \T$. Then $w \cdot \zeta \in q^{-1}(B_{q(\varepsilon)})$ if and only if $\zeta \in q^{-1}(B_{q(\varepsilon)})$. If $\zeta \in q^{-1}(B_{q(\varepsilon)})$, then $z_{w \cdot \zeta} = w \, z_\zeta$, and so $j_\varepsilon(w \cdot \zeta) = \overline{z_\varepsilon} \, w \, z_\zeta \, h_{q(\varepsilon)}(q(\zeta)) = w \, j_\varepsilon(\zeta)$. On the other hand, if $\zeta \notin q^{-1}(B_{q(\varepsilon)})$, then $j_\varepsilon(w \cdot \zeta) = 0 = w \, j_\varepsilon(\zeta)$. Hence $j_\varepsilon \in \Sigma_c(\IG;\IE)$.

We now show that $j_\varepsilon\restr{\IE_u} = \delta^\T_\varepsilon$. Fix $\zeta \in \IE_u$. Suppose first that $\zeta \notin q^{-1}(B_{q(\varepsilon)})$. Then $q(\zeta) \ne q(\varepsilon)$, and so $\zeta \ne w \cdot \varepsilon$ for all $w \in \T$. Hence $j_\varepsilon(\zeta) = 0 = \delta^\T_\varepsilon(\zeta)$. Now suppose that $\zeta \in q^{-1}(B_{q(\varepsilon)})$. Since $r\restr{B_{q(\varepsilon)}}(q(\zeta)) = q(u) = r\restr{B_{q(\varepsilon)}}(q(\varepsilon))$ and $B_{q(\varepsilon)}$ is a bisection, we have $q(\zeta) = q(\varepsilon)$. Hence $h_{q(\varepsilon)}(q(\zeta)) = 1$, and $\zeta = z_\zeta \cdot P_{q(\varepsilon)}(q(\varepsilon)) = (\overline{z_\varepsilon} \, z_\zeta) \cdot \varepsilon$. Therefore,
\[
j_\varepsilon(\zeta) = \overline{z_\varepsilon} \, z_\zeta \, h_{q(\varepsilon)}(q(\zeta)) = \overline{z_\varepsilon} \, z_\zeta = \delta^\T_\varepsilon(\zeta),
\]
and so $j_\varepsilon\restr{\IE_u} = \delta^\T_\varepsilon$. Finally, for all $w \in \T$, we have $q(w \cdot \varepsilon) = q(\varepsilon)$ and $\overline{z_{w \cdot \varepsilon}} = \overline{w} \, \overline{z_\varepsilon}$, and thus $j_{w \cdot \varepsilon} = \overline{w} \, j_\varepsilon$.
\end{proof}

We will use the following lemma to show that the map $\theta_u$ of \cref{thm: quotient maps}\cref{item: full quotient map} is injective.

\begin{lemma} \label{lemma: quotient unitaries}
Let $(\EE,i,q)$ be a twist over a Hausdorff \'etale groupoid $\GG$. Fix $u \in \EEo$, and let $J_u$ be the ideal of $C^*(\IG;\IE)$ defined in \cref{lemma: ideals for quotients}. For each $\varepsilon \in \IE_u$, define $\delta^\T_\varepsilon \in \Sigma_c\big(\IG_{q(u)};\IE_u\big)$ as in \cref{prop: delta^T_varepsilon}, and use \cref{lemma: j_varepsilon} to choose $j_\varepsilon \in \Sigma_c(\IG;\IE)$ such that $j_\varepsilon\restr{\IE_u} = \delta^\T_\varepsilon$. For any $\varepsilon \in \IE_u$ and $k_\varepsilon \in \Sigma_c(\IG;\IE)$ satisfying $k_\varepsilon\restr{\IE_u} = \delta^\T_\varepsilon$, we have $j_\varepsilon - k_\varepsilon \in J_u$. The quotient C*-algebra $C^*(\IG;\IE) / J_u$ is unital with identity $j_u + J_u$, and each $j_\varepsilon + J_u$ is a unitary element of $C^*(\IG;\IE) / J_u$. Moreover, $j_\varepsilon j_\zeta + J_u = j_{\varepsilon\zeta} + J_u$ for all $\varepsilon, \zeta \in \IE_u$.
\end{lemma}

\begin{proof}
For any $\varepsilon \in \IE_u$ and $k_\varepsilon \in \Sigma_c(\IG;\IE)$ satisfying $k_\varepsilon\restr{\IE_u} = \delta^\T_\varepsilon$, we have $(j_\varepsilon - k_\varepsilon)\restr{\IE_u} \equiv 0$, and hence $j_\varepsilon - k_\varepsilon \in J_u$.

We now show that $C^*(\IG;\IE) / J_u$ is unital. Fix $f \in \Sigma_c(\IG;\IE)$. For all $\xi \in \IE_u$, we have
\[
(j_u f)(\xi) = \int_{\IE_u} j_u(\eta) \, f(\eta^{-1} \xi) \, \d \lambda^u(\eta) = \int_\T \delta^\T_u(w \cdot u) \, f\big(\overline{w} \cdot (u\xi)\big) \, \d w = \int_\T w \, \overline{w} \, f(\xi) \, \d w = f(\xi).
\]
Hence $(j_u f - f)\restr{\IE_u} \equiv 0$, and so $j_u f - f \in J_u$. A similar argument shows that $f j_u - f \in J_u$. Thus
\[
(f + J_u)(j_u + J_u) \,=\, f j_u + J_u \,=\, f + J_u \,=\, j_u f + J_u \,=\, (j_u + J_u)(f + J_u),
\]
and so $j_u + J_u$ is the identity element of $C^*(\IG;\IE) / J_u$.

Fix $\varepsilon \in \IE_u$. We will show that $j_\varepsilon + J_u$ is a unitary element of $C^*(\IG;\IE) / J_u$. For all $\xi \in \IE_u$, we have
\begin{align*}
(j_\varepsilon j_\varepsilon^*)(\xi) \,= \int_{\IE_u} j_\varepsilon(\eta) \, \overline{j_\varepsilon(\xi^{-1}\eta)} \, \d \lambda^u(\eta) \,&= \int_\T \delta^\T_\varepsilon(w \cdot \varepsilon) \, \overline{\delta^\T_\varepsilon\big(w \cdot (\xi^{-1}\varepsilon)\big)} \, \d w \\
&= \int_\T w \, \overline{w} \, \overline{\delta^\T_\varepsilon(\xi^{-1}\varepsilon)} \, \d w \,=\, \overline{\delta^\T_u(\xi^{-1})} \,=\, j_u(\xi).
\end{align*}
Hence $(j_\varepsilon j_\varepsilon^* - j_u)\restr{\IE_u} \equiv 0$, and so $j_\varepsilon j_\varepsilon^* - j_u \in J_u$. Therefore,
\[
(j_\varepsilon + J_u)(j_\varepsilon + J_u)^* = j_\varepsilon j_\varepsilon^* + J_u = j_u + J_u,
\]
and a similar argument shows that $(j_\varepsilon + J_u)^*(j_\varepsilon + J_u) = j_u + J_u$. Thus $j_\varepsilon + J_u$ is a unitary.

Finally, fix $\varepsilon, \zeta \in \IE_u$. For all $\xi \in \IE_u$, we have
\begin{align*}
(j_\varepsilon j_\zeta)(\xi) \,= \int_{\IE_u} j_\varepsilon(\eta) \, j_\zeta(\eta^{-1}\xi) \, \d \lambda^u(\eta) \,&= \int_\T \delta^\T_\varepsilon(w \cdot \varepsilon) \, \delta^\T_\zeta\big(\overline{w} \cdot (\varepsilon^{-1}\xi)\big) \, \d w \\
&= \int_\T w \, \overline{w} \, \delta^\T_\zeta(\varepsilon^{-1}\xi) \, \d w \,=\, \delta^\T_\zeta(\varepsilon^{-1}\xi) \,=\, j_{\varepsilon\zeta}(\xi),
\end{align*}
and hence $(j_\varepsilon j_\zeta - j_{\varepsilon\zeta})\restr{\IE_u} \equiv 0$. Thus $j_\varepsilon j_\zeta - j_{\varepsilon\zeta} \in J_u$, and so $j_\varepsilon j_\zeta + J_u = j_{\varepsilon\zeta} + J_u$.
\end{proof}

\begin{proof}[Proof of \cref{thm: quotient maps}]
Fix $f \in \Sigma_c(\IG;\IE)$. Then $f\restr{\IE_u}$ is continuous. We claim that $f\restr{\IE_u} \in \Sigma_c\big(\IG_{q(u)};\IE_u\big)$. We have $\osupp(f\restr{\IE_u}) = \osupp(f) \medcap \IE_u$, and so since $\IE_u = r\restr{\IE}^{-1}(u)$ is closed in $\IE$, $\supp(f) \medcap \IE_u$ is a compact subset of $\IE_u$. Hence $\supp(f\restr{\IE_u}) \subseteq \supp(f)$, and so $f\restr{\IE_u} \in C_c(\IE_u)$. For all $\varepsilon \in \IE$ and $z \in \T$, we have $z \cdot \varepsilon \in \IE_u$ if and only if $\varepsilon \in \IE_u$, and so $f\restr{\IE_u}(z \cdot \varepsilon) = z f\restr{\IE_u}(\varepsilon)$ for all $\varepsilon \in \IE_u$ and $z \in \T$. Hence $f\restr{\IE_u} \in \Sigma_c\big(\IG_{q(u)};\IE_u\big)$.

For part~\cref{item: reduced quotient map}, define $\psi_u^r\colon \Sigma_c(\IG;\IE) \to C_r^*(\IG_{q(u)}; \IE_u)$ by $\psi_u^r(f) \coloneqq f\restr{\IE_u}$. Routine calculations show that $\psi_u^r$ is a $*$-homomorphism that vanishes on $F_u$. We now show that $\psi_u^r$ is bounded in the reduced norm. Fix $f \in \Sigma_c(\IG;\IE)$. For each $v \in \EEo$, let
\[
\pi_v^\II\colon \Sigma_c(\IG;\IE) \to B\big(L^2\big(\IG_{q(v)};\IE_v\big)\big) \quad \text{and} \quad \rho^{\IE_v}\colon \Sigma_c\big(\IG_{q(v)};\IE_v\big) \to B\big(L^2\big(\IG_{q(v)};\IE_v\big)\big)
\]
be the regular representations associated to $v$ of $\Sigma_c(\IG;\IE)$ and $\Sigma_c\big(\IG_{q(v)};\IE_v\big)$, respectively, onto the space of square-integrable $\T$-equivariant functions on $\IE_v$. For each $v \in \EEo$, we have $\pi_v^\II(f) = \rho^{\IE_v}\big(f\restr{\IE_v}\big)$, and hence
\[
\lv \psi_u^r(f) \rv_r = \lv f\restr{\IE_u} \rv_r = \big\lv \rho^{\IE_v}\big(f\restr{\IE_u}\big) \big\rv \le \sup\left\{ \big\lv \pi_v^\II(f) \big\rv : v \in \EEo \right\} = \lv f \rv_r.
\]
So $\psi_u^r$ is bounded, and since $\Sigma_c(\IG;\IE)$ is dense in $C_r^*(\IG;\IE)$, $\psi_u^r$ extends to a C*-homomorphism $\overline{\psi_u^r}\colon C_r^*(\IG;\IE) \to C_r^*(\IG_{q(u)}; \IE_u)$. Recall from \cref{lemma: ideals for quotients} that $J_u^r$ is an ideal of $C_r^*(\IG;\IE)$. Since $\overline{\psi_u^r}$ is bounded and vanishes on $F_u$ by definition, it vanishes on $J_u^r$. Therefore, $\overline{\psi_u^r}$ descends to a $*$-homomorphism
\[
\theta_u^r\colon C_r^*(\IG;\IE) / J_u^r \to C_r^*(\IG_{q(u)}; \IE_u)
\]
satisfying $\theta_u^r(f + J_u^r) = f\restr{\IE_u}$ for all $f \in \Sigma_c(\IG;\IE)$. We claim that $\theta_u^r$ is surjective. Since $\Sigma_c\big(\IG_{q(u)};\IE_u\big)$ is dense in $C_r^*(\IG_{q(u)}; \IE_u)$, it suffices to show that $\Sigma_c\big(\IG_{q(u)};\IE_u\big)$ is contained in the image of $\theta_u^r$. Recall from \cref{prop: delta^T_varepsilon} that
\[
\Sigma_c\big(\IG_{q(u)};\IE_u\big) = \vecspan\!\left\{ \delta^\T_\varepsilon : \varepsilon \in \IE_u \right\}\!.
\]
For each $\varepsilon \in \IE_u$, use \cref{lemma: j_varepsilon} to choose $j_\varepsilon \in \Sigma_c(\IG;\IE)$ such that $j_\varepsilon\restr{\IE_u} = \delta^\T_\varepsilon$. Then for each $\varepsilon \in \IE_u$, we have $\theta_u^r(j_\varepsilon + J_u^r) = j_\varepsilon\restr{\IE_u} = \delta^\T_\varepsilon$, and hence $\theta_u^r$ is surjective.

For part~\cref{item: full quotient map}, define $\psi_u\colon \Sigma_c(\IG;\IE) \to C^*(\IG_{q(u)}; \IE_u)$ by $\psi_u(f) \coloneqq f\restr{\IE_u}$. The argument used in part~\cref{item: reduced quotient map} shows that $\psi_u$ is a $*$-homomorphism that vanishes on $F_u$. Thus \cref{lemma: *-hom extension to full C*} implies that $\psi_u$ extends uniquely to a C*-homomorphism $\overline{\psi_u}\colon C^*(\IG;\IE) \to C^*(\IG_{q(u)}; \IE_u)$. By \cref{lemma: ideals for quotients}, $J_u$ is an ideal of $C^*(\IG;\IE)$, and so a similar argument to the one used in part~\cref{item: reduced quotient map} shows that there is a surjective $*$-homomorphism
\[
\theta_u\colon C^*(\IG;\IE) / J_u \to C^*(\IG_{q(u)}; \IE_u)
\]
satisfying $\theta_u(f + J_u) = f\restr{\IE_u}$ for all $f \in \Sigma_c(\IG;\IE)$. To see that $\theta_u$ is injective, recall from \cref{lemma: j_varepsilon} that for each $\varepsilon \in \IE_u$, there exists $j_\varepsilon \in \Sigma_c(\IG;\IE)$ such that $j_\varepsilon\restr{\IE_u} = \delta^\T_\varepsilon$ and $j_{z \cdot \varepsilon} = \overline{z} \, j_\varepsilon$ for each $z \in \T$. By \cref{lemma: quotient unitaries}, $\big\{ j_\varepsilon + J_u : \varepsilon \in \IE_u \big\}$ is a family of unitary elements of $C^*(\IG;\IE) / J_u$ such that $(j_\varepsilon + J_u)(j_\zeta + J_u) = j_{\varepsilon\zeta} + J_u$ and $j_{z \cdot \varepsilon} + J_u = \overline{z}(j_\varepsilon + J_u)$ for all $\varepsilon, \zeta \in \IE_u$ and $z \in \T$. Hence \cref{thm: universal property} implies that there is a homomorphism
\[
\eta_u\colon C^*(\IG_{q(u)}; \IE_u) \to C^*(\IG;\IE) / J_u
\]
satisfying $\eta_u\big(\delta^\T_\varepsilon\big) = j_\varepsilon + J_u$ for each $\varepsilon \in \IE_u$, where $\delta^\T_\varepsilon \in \Sigma_c\big(\IG_{q(u)};\IE_u\big)$ is defined as in \cref{prop: delta^T_varepsilon}. We claim that $\eta_u \circ \theta_u$ is the identity map on $C^*(\IG;\IE) / J_u$. To see this, observe that since $\left\{ f + J_u : f \in \Sigma_c(\IG;\IE) \right\}$ is a dense subspace of $C^*(\IG;\IE) / J_u$ and since $\eta_u \circ \theta_u$ is continuous, it suffices to show that $\eta_u\big(\theta_u(f + J_u)\big) = f + J_u$ for all $f \in \Sigma_c(\IG;\IE)$. Fix $f \in \Sigma_c(\IG;\IE)$. Since
\[
\theta_u(f + J_u) = f\restr{\IE_u} \in \Sigma_c\big(\IG_{q(u)};\IE_u\big) = \vecspan\!\left\{ \delta^\T_\varepsilon : \varepsilon \in \IE_u \right\}\!,
\]
there is a finite subset $F$ of $\IE_u$ and a collection $\{ c_\varepsilon \in \C {\setminus} \{0\} : \varepsilon \in F \}$ such that $f\restr{\IE_u} = \sum_{\varepsilon \in F} c_\varepsilon \, \delta^\T_\varepsilon$. Since $\eta_u$ is linear, we have
\begin{equation} \label{eqn: eta_u circ theta_u is identity map}
\eta_u\big(\theta_u(f + J_u)\big) \,=\, \eta_u\big(f\restr{\IE_u}\big) \,=\, \eta_u\Big(\sum_{\varepsilon \in F} c_\varepsilon \, \delta^\T_\varepsilon\Big) \,=\, \sum_{\varepsilon \in F} c_\varepsilon \, \eta_u\big(\delta^\T_\varepsilon\big) \,=\, \Big(\sum_{\varepsilon \in F} c_\varepsilon \, j_\varepsilon\Big) + J_u.
\end{equation}
Since $\big(f - \sum_{\varepsilon \in F} c_\varepsilon \, j_\varepsilon\big)\restr{\IE_u} = f\restr{\IE_u} - \sum_{\varepsilon \in F} c_\varepsilon \, \delta^\T_\varepsilon \equiv 0$, we have $f - \sum_{\varepsilon \in F} c_\varepsilon \, j_\varepsilon \in J_u$. Thus we deduce from \cref{eqn: eta_u circ theta_u is identity map} that $\eta_u\big(\theta_u(f + J_u)\big) = f + J_u$. Therefore, $\eta_u \circ \theta_u$ is the identity map on $C^*(\IG;\IE) / J_u$, and so $\theta_u$ is injective.
\end{proof}

\begin{cor} \label{cor: quotient map on unit space algebra}
Let $(\EE,i,q)$ be a twist over a Hausdorff \'etale groupoid $\GG$. Fix $u \in \EEo$. Recall from \cref{prop: delta^T_varepsilon} the definition of the identity element $\delta^\T_u$ of both $C_r^*(\IG;\IE)$ and $C^*(\IG;\IE)$. Recall from \cref{thm: quotient maps} the definitions of the ideals $J_u^r$ of $C_r^*(\IG;\IE)$ and $J_u$ of $C^*(\IG;\IE)$, and the maps
\[
\theta_u^r\colon C_r^*(\IG;\IE) / J_u^r \to C_r^*(\IG_{q(u)}; \IE_u) \quad \text{ and } \quad \theta_u\colon C^*(\IG;\IE) / J_u \to C^*(\IG_{q(u)}; \IE_u).
\]
\begin{enumerate}[label=(\alph*)]
\item \label{item: reduced quotient map on unit space algebra} The map $Q_u^r\colon C_r^*(\IG;\IE) \to C_r^*(\IG_{q(u)}; \IE_u)$ given by $Q_u^r(a) \coloneqq \theta_u^r(a + J_u^r)$ is a surjective $*$-homomorphism. For every $g \in \Sigma_c(\IG;\IE)$ satisfying $q(\supp(g)) \subseteq \GGo$ and $g(u) = 1$, we have $Q_u^r(g) = \delta_u^\T$.
\item \label{item: full quotient map on unit space algebra} The map $Q_u\colon C^*(\IG;\IE) \to C^*(\IG_{q(u)}; \IE_u)$ given by $Q_u(a) \coloneqq \theta_u(a + J_u)$ is a surjective $*$-homomorphism. For every $g \in \Sigma_c(\IG;\IE)$ satisfying $q(\supp(g)) \subseteq \GGo$ and $g(u) = 1$, we have $Q_u(g) = \delta_u^\T$.
\end{enumerate}
\end{cor}

\begin{proof}
We will only prove part~\cref{item: reduced quotient map on unit space algebra}, as the proof of part~\cref{item: full quotient map on unit space algebra} is identical. The map $Q_u^r$ is a surjective $*$-homomorphism because it is the composition of $\theta_u^r$ and the quotient map from $C_r^*(\IG;\IE)$ to $C_r^*(\IG;\IE) / J_u^r$. Suppose that $g \in \Sigma_c(\IG;\IE)$ satisfies $q(\supp(g)) \subseteq \GGo$ and $g(u) = 1$. By the definition of $Q_u^r$ and $\theta_u^r$, we have $Q_u^r(g) = \theta_u^r(g + J_u^r) = g\restr{\IE_u}$. So we must show that $g\restr{\IE_u} = \delta^\T_u$. Fix $\varepsilon \in \IE_u$. If $\varepsilon \in q^{-1}(\GGo)$, then $q(\varepsilon) \in \GGo \cap \IG_{q(u)}$, and so $q(\varepsilon) = q(u)$. Thus $\varepsilon = z_\varepsilon \cdot u = i(q(u), z_\varepsilon)$ for some unique $z_\varepsilon \in \T$, and so since $g$ is $\T$-equivariant, we have $g(\varepsilon) = g(z_\varepsilon \cdot u) = z_\varepsilon \, g(u) = z_\varepsilon = \delta^\T_u(\varepsilon)$. If $\varepsilon \notin q^{-1}(\GGo)$, then $\varepsilon \ne z \cdot u$ for all $z \in \T$, and so $g(\varepsilon) = 0 = \delta^\T_u(\varepsilon)$. Therefore, $g\restr{\IE_u} = \delta^\T_u$, as required.
\end{proof}

In our proof of \cref{thm: quotient maps}\cref{item: full quotient map}, we used the universal property of $C^*(\IG_{q(u)}; \IE_u)$ given in \cref{thm: universal property} to show that $\theta_u\colon C^*(\IG;\IE) / J_u \to C^*(\IG_{q(u)}; \IE_u)$ is injective, but this argument doesn't work in the reduced setting because the universal property doesn't hold. In fact, somewhat surprisingly, even though $\theta_u$ is always an isomorphism, there exist examples of groupoids for which the map $\theta_u^r\colon C_r^*(\IG;\IE) / J_u^r \to C_r^*(\IG_{q(u)}; \IE_u)$ of \cref{thm: quotient maps}\cref{item: reduced quotient map} is \emph{not} an isomorphism. One such example, due to Willett \cite{Willett2015}, comes from the class of \emph{HLS groupoids} constructed in \cite[Section~2]{HLS2002} (see also \cite[Definition~2.2]{Willett2015}). Before presenting Willett's example, we first recall the construction of an HLS groupoid.

Suppose that $(K_n)_{n \in \N}$ is an approximating sequence for a discrete group $\Gamma$, as defined in \cite[Definition~2.1]{Willett2015}. For each $n \in \N$, define $\Gamma_n \coloneqq \Gamma / K_n$, and $\Gamma_\infty \coloneqq \Gamma$. For each $n \in \N \cup \{\infty\}$, denote the identity of the group $\Gamma_n$ by $e_{\Gamma_n}$. Define
\[
\GG = \bigsqcup_{n \in \N \cup \{\infty\}} \{n\} \times \Gamma_n,
\]
and equip $\GG$ with the groupoid operations coming from the group structure on the fibres over each $n \in \N \cup \{\infty\}$. Then $\GGo = \{ (n,g) \in \GG : g = e_{\Gamma_n} \}$, and for all $(n,g) \in \GG$, we have $r(n,g) = (n,e_{\Gamma_n}) = s(n,g)$. Thus $\Iso(\GG) = \GG$. In \cite[Definition~2.2]{Willett2015}, Willett endows this groupoid $\GG$ with a topology, under which it is a second-countable locally compact Hausdorff \'etale groupoid, called the \hl{HLS groupoid} associated to the approximated group $(\Gamma, (K_n)_{n \in \N})$.

\begin{example} \label{eg: Willett}
Let $F_2$ denote the free group on two generators. For each $n \in \N$, define
\begin{equation} \label{eqn: definition of K_n}
K_n \coloneqq \bigcap \, \{ \ker(\phi) \,:\, \phi\colon F_2 \to \Gamma \text{ is a group homomorphism},\, \lav \Gamma \rav \le n \}.
\end{equation}
By \cite[Lemma~2.8]{Willett2015}, $(K_n)_{n \in \N}$ is an approximating sequence for $F_2$. Since $F_2$ is not amenable, \cite[Lemma~2.4]{Willett2015} shows that the HLS groupoid $\GG$ associated to the approximated group $(F_2, (K_n)_{n \in \N})$ is not amenable. However, by \cite[Lemmas 2.7 and 2.8]{Willett2015}, $C^*(\GG)$ is isomorphic to $C_r^*(\GG)$.
\end{example}

Taking $\EE$ to be the trivial twist $\GG \times \T$ over the HLS groupoid $\GG$ described in \cref{eg: Willett}, and $u \coloneqq \big((\infty, e_{F_2}), 1\big) \in \EEo$, we now prove that the map $\theta_u^r\colon C_r^*(\IG;\IE) / J_u^r \to C_r^*(\IG_{q(u)}; \IE_u)$ of \cref{thm: quotient maps}\cref{item: reduced quotient map} is not injective. Note that since $\EE$ is trivial, we omit it from our notation in \cref{thm: reduced quotient map not injective}, and we identify $\Sigma_c(\IG;\IE)$ with $C_c(\IG)$ and $C_r^*(\IG;\IE)$ with $C_r^*(\IG)$ via the $*$-isomorphism defined in \cref{prop: C*s of twists and 2-cocycles}.

\begin{theorem} \label{thm: reduced quotient map not injective}
For each $n \in \N$, let $K_n$ be defined as in \cref{eqn: definition of K_n}, and let $\GG$ be the HLS groupoid associated to $(F_2, (K_n)_{n\in\N})$, as described in \cref{eg: Willett}. Consider the unit $u \coloneqq (\infty, e_{F_2}) \in \GGo$, where $e_{F_2}$ is the identity element of $F_2$. The map $\theta_u^r\colon C_r^*(\IG) / J_u^r \to C_r^*(\IG_u)$ of \cref{thm: quotient maps}\cref{item: reduced quotient map} is not injective.
\end{theorem}

\begin{proof}
Define $\Gamma \coloneqq F_2$. Since $\Iso(\GG) = \GG$, we have $\IG_u = \GG_u^u = \{\infty\} \times \Gamma_\infty \cong \Gamma$. Let $\Upsilon\colon C^*(\IG) \to C_r^*(\IG)$ be the unique homomorphism that restricts to the identity map on $C_c(\IG)$. By \cite[Lemmas 2.7 and 2.8]{Willett2015} and \cite[Corollary~3.3.23]{Armstrong2019}, $\Upsilon$ is an isomorphism, and so the full and reduced C*-norms coincide on $C_c(\IG)$. Thus $\Upsilon\restr{J_u}\colon J_u \to J_u^r$ is an isomorphism. Observe that for all $a, b \in C^*(\IG)$ with $a - b \in J_u$, we have $\Upsilon(a) - \Upsilon(b) = \Upsilon(a - b) \in J_u^r$. Hence there is a map $\Upsilon_u\colon C^*(\IG) / J_u \to C_r^*(\IG) / J_u^r$ satisfying $\Upsilon_u(a + J_u) = \Upsilon(a) + J_u^r$, for all $a \in C^*(\IG)$. Since $\Upsilon\colon C^*(\IG) \to C_r^*(\IG)$ and $\Upsilon\restr{J_u}\colon J_u \to J_u^r$ are isomorphisms, it is clear that $\Upsilon_u$ is also an isomorphism.

Let $\Xi_u\colon C^*(\IG_u) \to C_r^*(\IG_u)$ be the unique homomorphism that restricts to the identity map on $C_c(\IG_u)$. Since $\Gamma = F_2$ is a nonamenable group, we know that $\Xi_u$ is not injective. For all $f \in C_c(\IG_u)$, we have
\[
\theta_u^r(\Upsilon_u(f + J_u)) = \theta_u^r(\Upsilon(f) + J_u^r) = \theta_u^r(f + J_u^r) = f\restr{\IG_u} = \Xi_u(f\restr{\IG_u}) = \Xi_u(\theta_u(f + J_u)).
\]
Therefore, the maps $\theta_u^r \circ \Upsilon_u$ and $\Xi_u \circ \theta_u$ agree on the set $\{ f + J_u : f \in C_c(\IG) \}$, which is a dense subspace of $C^*(\IG) / J_u$. Since these maps are homomorphisms of C*-algebras, they are continuous, and hence they agree on all of $C^*(\IG) / J_u$. Therefore, $\theta_u^r = \Xi_u \circ \theta_u \circ \Upsilon_u^{-1}$. Since $\theta_u \circ \Upsilon_u^{-1}$ is surjective and $\Xi_u$ is not injective, it follows that $\theta_u^r$ is not injective.
\end{proof}

\section{Unique state extensions}
\label{sec: states}

In this short section we provide a sufficient \emph{compressibility} condition under which a (pure) state of a C*-subalgebra $B$ of a (not necessarily unital) C*-algebra $A$ has a unique (pure) state extension to $A$. This result is an extension of a well known result of Anderson about unital C*-algebras (proved in the paragraph preceding \cite[Theorem~3.2]{Anderson1979}), but we were unable to find a proof of it in the literature, and so we present one here.

We begin by introducing some notation.

\begin{notation} \label{notation: state sets}
Given a state $\phi$ of a C*-algebra $A$, we define
\[
\MM_\phi \coloneqq \{ a \in A : \phi(ax) = \phi(xa) = \phi(a)\phi(x) \text{ for all } x \in A \}
\]
and
\[
\UU_\phi \coloneqq \{ a \in A : \lav \phi(a) \rav = \lv a \rv = 1 \}.
\]
\end{notation}

\begin{remark} \label{rem: U_phi subseteq M_phi}
At the bottom of page~304 of \cite{Anderson1979}, Anderson observes that if $\phi$ is a state of a unital C*-algebra $A$, then $\UU_\phi \subseteq \MM_\phi$. In fact, the same is true when $A$ is nonunital. To see this, let $\phi^+$ be the unique state extension of $\phi$ to the minimal unitisation $A^+$ of $A$. Then for all $a \in \UU_\phi$, we have $(a,0) \in \UU_{\phi^+} \subseteq \MM_{\phi^+}$, and it follows that $a \in \MM_\phi$.
\end{remark}

We now recall Anderson's compressibility condition given in \cite[Section~3]{Anderson1979}.

\begin{definition} \label{def: compressible}
Let $A$ be a C*-algebra and let $B$ be a C*-subalgebra of $A$. Suppose that $\phi$ is a state of $B$. We say that $A$ is \hl{$B$-compressible modulo $\phi$} if, for each $a \in A$ and $\epsilon > 0$, there exists $b \in \UU_\phi \subseteq B$ and $c \in B$ such that $b \ge 0$ and $\lv bab - c \rv < \epsilon$.
\end{definition}

We thank the anonymous referee for providing a more elegant proof of the following theorem than the one that appeared in the initial preprint of this article.

\begin{thm} \label{thm: unique state extensions}
Let $A$ be a C*-algebra and let $B$ be a C*-subalgebra of $A$. If $\phi$ is a state of $B$ such that $A$ is $B$-compressible modulo $\phi$, then $\phi$ has a unique state extension to $A$. If $\phi$ is a pure state, then so is its unique extension.
\end{thm}

\begin{proof}
Let $\phi$ be a state of $B$ such that $A$ is $B$-compressible modulo $\phi$. Let $B^+$ denote the minimal unitisation of $B$, and let $\phi^+$ be the unique state extension of $\phi$ to $B^+$, which is given by $\phi^+(b,\lambda) \coloneqq \phi(b) + \lambda$ for all $(b,\lambda) \in B \times \C$. We claim that $A^+$ is $B^+$-compressible modulo $\phi^+$. To see this, fix $(a,\lambda) \in A^+$ and $\epsilon > 0$. Then $a \in A$, and since $A$ is $B$-compressible modulo $\phi$, there exists $b \in \UU_\phi \subseteq B$ and $c \in B$ such that $b \ge 0$ and $\lv bab - c \rv < \epsilon$. Let $b' \coloneqq (b,0) \in B^+$ and $c' \coloneqq (\lambda b^2 + c, 0) \in B^+$. Then $b' \in \UU_{\phi^+}$ and $b' \ge 0$, and we have
\[
\lv b' (a,\lambda) b' - c' \rv = \lv (bab + \lambda b^2, 0) - (\lambda b^2 + c, 0) \rv = \lv (bab - c, 0) \rv = \lv bab - c \rv < \epsilon,
\]
as claimed. Thus \cite[Theorem~3.2]{Anderson1979} implies that $\phi^+$ extends uniquely to a state $\overline{\phi^+}$ of $A^+$. By identifying $A$ and $B$ with their images under the isometric inclusions into $A^+$ and $B^+$, respectively, we obtain a state $\overline{\phi} \coloneqq \overline{\phi^+}\restr{A}$ of $A$ that satisfies $\overline{\phi}\restr{B} = \phi^+\restr{B} = \phi$.

To see that $\overline{\phi}$ is unique, suppose that $\psi$ is a state of $A$ that satisfies $\psi\restr{B} = \phi$. Since $\overline{\phi^+}$ is the unique state extension of $\phi$ to $A^+$, it must also be the unique state extension of $\psi$ to $A^+$, and hence $\psi = \phi$.

Finally, if $\phi$ is a pure state of $B$, then \cite[II.6.3.2]{Blackadar2006} implies that the unique state extension of $\phi$ to $A$ is also a pure state.
\end{proof}

\section{A uniqueness theorem for reduced twisted groupoid \texorpdfstring{C*}{C*}-algebras}
\label{sec: uniqueness}

In this section we prove that there is an embedding $\iota_r$ of the twisted C*-algebra $C_r^*(\IG;\IE)$ associated to the interior of the isotropy of a Hausdorff \'etale groupoid $\GG$ into $C_r^*(\GG;\EE)$ (see \cref{prop: injective homo}). We use this result to prove our uniqueness theorem (\cref{thm: uniqueness theorem}), which states that a C*-homomorphism $\Psi$ of $C_r^*(\GG;\EE)$ is injective if and only if $\Psi \circ \iota_r$ is injective. We then use our uniqueness theorem to prove \cref{cor: simplicity}, which states that if $\GG$ is effective, then $C_r^*(\GG;\EE)$ is simple if and only if $\GG$ is minimal. With the exception of \cref{cor: simplicity}, the results in this section are extensions of the results in \cite[Section~5.3]{Armstrong2019} to the setting of C*-algebras of groupoid twists.

\begin{prop} \label{prop: injective homo}
Let $(\EE,i,q)$ be a twist over a Hausdorff \'etale groupoid $\GG$. There is a homomorphism $\iota\colon C^*(\IG;\IE) \to C^*(\GG;\EE)$ such that
\[
\iota(f)(\varepsilon) =
\begin{cases}
f(\varepsilon) & \text{if } \varepsilon \in \IE \\
0 & \text{if } \varepsilon \notin \IE
\end{cases}
\]
for all $f \in \Sigma_c(\IG;\IE)$ and $\varepsilon \in \EE$. We have $\iota\big(\Sigma_c(\IG;\IE)\big) \subseteq \Sigma_c(\GG;\EE)$, and $\iota$ descends to an injective homomorphism $\iota_r\colon C_r^*(\IG;\IE) \to C_r^*(\GG;\EE)$. If $\IG$ is amenable, then $\iota$ is also injective.
\end{prop}

\begin{proof}
A standard argument shows that $\iota\restr{\Sigma_c(\IG;\IE)}$ is a $*$-homomorphism with image contained in $\Sigma_c(\GG;\EE)$, and hence \cref{lemma: *-hom extension to full C*} shows that it extends uniquely to a C*-homomorphism $\iota\colon C^*(\IG;\IE) \to C^*(\GG;\EE)$. Since $\IE$ is an open subgroupoid of $\GG$ and $\IE = q^{-1}(\IG)$ by the proof of \cref{cor: specific subgroupoid twists}\cref{item: isotropy subgroupoid twist}, \cite[Lemma~2.7]{BFPR2021} implies that $\iota$ descends to an injective homomorphism $\iota_r\colon C_r^*(\IG;\IE) \to C_r^*(\GG;\EE)$. For the final claim, suppose that $\IG$ is amenable, and fix $f \in \Sigma_c(\IG;\IE)$. Using \cite[Theorem~11.1.11]{Sims2020} for the first equality and that $\iota_r$ is isometric for the second equality, we see that
\[
\lv f \rv = \lv f \rv_r = \lv \iota_r(f) \rv_r = \lv \iota(f) \rv_r \le \lv \iota(f) \rv \le \lv f \rv,
\]
and hence $\iota$ is injective.
\end{proof}

We now prove that if $\IG$ is closed, then the map that restricts functions in $\Sigma_c(\GG;\EE)$ to $\IE$ extends to a conditional expectation from $C_r^*(\GG;\EE)$ to $\iota_r\big(C_r^*(\IG;\IE)\big)$, and that it also extends to a conditional expectation from $C^*(\GG;\EE)$ to $\iota\big(C^*(\IG;\IE)\big)$ if $\IG$ is amenable.

Although we do not actually use this result in this article, we include it here as we believe it could have various future applications; for example, for constructing KK-classes that may be valuable in K-theory calculations or noncommutative geometry. Moreover, the result has already been used in \cite[Section~6]{ABS2022} in order to prove a simplicity characterisation for twisted C*-algebras of Deaconu--Renault groupoids.

Note that the hypothesis of $\IG$ being closed in $\GG$ is indeed necessary, because there do exist Hausdorff \'etale groupoids for which this is not the case (see, for instance, \cite[Example~4.7]{BNRSW2016}), and because \cite[Lemma~3.4]{BEFPR2021} shows that if $\IG$ is not closed, then there is no conditional expectation $C_r^*(\GG;\EE) \to \iota_r\big(C_r^*(\IG;\IE)\big)$.

\begin{lemma}
Let $\GG$ be a Hausdorff \'etale groupoid such that $\IG$ is closed in $\GG$, and let $(\EE,i,q)$ be a twist over $\GG$. For all $f \in \Sigma_c(\GG;\EE)$, we have $f\restr{\IE} \in \Sigma_c(\IG;\IE)$.
\begin{enumerate}[label=(\alph*)]
\item \label{item: reduced conditional expectation} Let $\iota_r\colon C_r^*(\IG;\IE) \to C_r^*(\GG;\EE)$ be the homomorphism of \cref{prop: injective homo}. There is a conditional expectation $\Psi_r^\II\colon C_r^*(\GG;\EE) \to \iota_r\big(C_r^*(\IG;\IE)\big)$ satisfying $\Psi_r^\II(f) = \iota_r(f\restr{\IE})$ for all $f \in \Sigma_c(\GG;\EE)$, and $\Psi_r^\II \circ \iota_r = \iota_r$.
\item \label{item: full conditional expectation} Let $\iota\colon C^*(\IG;\IE) \to C^*(\GG;\EE)$ be the homomorphism of \cref{prop: injective homo}. If $\IG$ is amenable, then there is a conditional expectation $\Psi^\II\colon C^*(\GG;\EE) \to \iota\big(C^*(\IG;\IE)\big)$ satisfying $\Psi^\II(f) = \iota(f\restr{\IE})$ for all $f \in \Sigma_c(\GG;\EE)$, and $\Psi^\II \circ \iota = \iota$.
\end{enumerate}
\end{lemma}

\begin{proof}
Since $\IE = q^{-1}(\IG)$ is closed, it follows that for all $f \in \Sigma_c(\GG;\EE)$, we have $f\restr{\IE} \in \Sigma_c(\IG;\IE)$. Since $\IG$ is clopen in $\GG$, part~\cref{item: reduced conditional expectation} follows immediately from \cite[Lemma~3.4]{BEFPR2021}.

For part~\cref{item: full conditional expectation}, we follow the argument used to prove \cite[Proposition~4.1(c)]{BNRSW2016}. Let $M \coloneqq \iota\big(C^*(\IG;\IE)\big)$, and define $\Psi^\II\colon \Sigma_c(\GG;\EE) \to M$ by $\Psi^\II(f) \coloneqq \iota(f\restr{\IE})$. It is clear that $\Psi^\II$ is linear. To see that $\Psi^\II$ is bounded, fix $f \in \Sigma_c(\GG;\EE)$. Since $\IG$ is amenable and $f\restr{\IE} \in \Sigma_c(\IG;\IE)$, \cite[Theorem~11.1.11]{Sims2020} implies that $\lv f\restr{\IE} \rv = \lv f\restr{\IE} \rv_r$. Thus, since $\iota$ and $\iota_r$ are isometric by \cref{prop: injective homo} and since $\Psi_r^\II$ is bounded by part~\cref{item: reduced conditional expectation}, we have
\[
\big\lv \Psi^\II(f) \big\rv = \lv \iota(f\restr{\IE}) \rv = \lv f\restr{\IE} \rv = \lv f\restr{\IE} \rv_r = \lv \iota_r(f\restr{\IE}) \rv_r = \big\lv \Psi_r^\II(f) \big\rv_r \le \lv f \rv_r \le \lv f \rv,
\]
as required. For all $f \in \iota\big(\Sigma_c(\IG;\IE)\big)$, we have $\Psi^\II(f) = \iota(f\restr{\IE}) = f$. Since $\iota\big(\Sigma_c(\IG;\IE)\big)$ is dense in $M$, it follows that $\Psi^\II(f) = f$ for all $f \in M$, and so $\Psi^\II \circ \iota = \iota$. Hence $\Psi^\II$ is a projection, and since $M$ is nontrivial, we have $\lv \Psi^\II \rv \ge 1$. Therefore, \cite[Theorem~II.6.10.2]{Blackadar2006} implies that $\Psi^\II$ is a conditional expectation.
\end{proof}

We now present our main result: a uniqueness theorem for reduced twisted groupoid C*-algebras, which generalises \cite[Theorem~3.1(b)]{BNRSW2016}.

\begin{thm}[C*-uniqueness theorem] \label{thm: uniqueness theorem}
Let $(\EE,i,q)$ be a twist over a Hausdorff \'etale groupoid $\GG$. Let $\iota_r\colon C_r^*(\IG;\IE) \to C_r^*(\GG;\EE)$ be the injective homomorphism of \cref{prop: injective homo}, and define $M_r \coloneqq \iota_r\big(C_r^*(\IG;\IE)\big)$. Suppose that $A$ is a C*-algebra and that $\Psi\colon C_r^*(\GG;\EE) \to A$ is a C*-homomorphism. Then $\Psi$ is injective if and only if $\Psi \circ \iota_r$ is an injective C*-homomorphism of $C_r^*(\IG;\IE)$.
\end{thm}

In order to prove \cref{thm: uniqueness theorem}, we need the following preliminary result, which is an extension of \cite[Theorem~3.1(a)]{BNRSW2016} to the twisted setting. Recall that, given C*-algebras $A$ and $B$ and a surjective $*$-homomorphism $Q\colon A \to B$, we say that a state $\psi$ of $A$ \hl{factors through} $B$ if there exists a state $\phi$ of $B$ such that $\psi = \phi \circ Q$. It follows from \cite[Theorems~3.3.1~and~3.3.3]{Murphy1990} that in this setting, a state $\psi$ of $A$ factors through $B$ if and only if $\ker(Q) \subseteq \ker(\psi)$.

\begin{prop} \label{prop: unique extensions of states of M and M_r}
Let $(\EE,i,q)$ be a twist over a Hausdorff \'etale groupoid $\GG$. Let
\[
\iota_r\colon C_r^*(\IG;\IE) \to C_r^*(\GG;\EE) \quad \text{ and } \quad \iota\colon C^*(\IG;\IE) \to C^*(\GG;\EE)
\]
be the homomorphisms of \cref{prop: injective homo}, and define
\[
M_r \coloneqq \iota_r\big(C_r^*(\IG;\IE)\big) \quad \text{ and } \quad M \coloneqq \iota\big(C^*(\IG;\IE)\big).
\]
Suppose that $u \in \EEo$ satisfies $\EE_u^u = \IE_u$.
\begin{enumerate}[label=(\alph*)]
\item \label{item: unique extensions reduced} If $\varphi_r$ is a state of $M_r$ such that $\varphi_r \circ \iota_r$ factors through $C_r^*\big(\GG_{q(u)}^{q(u)};\EE_u^u\big)$, then $\varphi_r$ has a unique state extension to $C_r^*(\GG;\EE)$.
\item If $\varphi$ is a state of $M$ such that $\varphi \circ \iota$ factors through $C^*\big(\GG_{q(u)}^{q(u)};\EE_u^u\big)$, then $\varphi$ has a unique state extension to $C^*(\GG;\EE)$.
\end{enumerate}
\end{prop}

In order to prove \cref{prop: unique extensions of states of M and M_r}, we need the following two preliminary results. The first of these results is a generalisation of \cite[Lemma~3.3(b)]{BNRSW2016} to the twisted setting.

\begin{lemma} \label{lemma: element for compression}
Let $(\EE,i,q)$ be a twist over a Hausdorff \'etale groupoid $\GG$. Suppose that $u \in \EEo$ satisfies $\EE_u^u = \IE_u$. For each $f \in \Sigma_c(\GG;\EE)$, there exists $g \in \Sigma_c(\IG;\IE)$ satisfying $g \ge 0$, $q(\supp(g)) \subseteq \GGo$, $\lv g \rv = \lv g \rv_r = g(u) = 1$, and $\supp(gfg) \subseteq \IE$.
\end{lemma}

\begin{proof}
First observe that since $\EE_u^u = \IE_u$, we have $\GG_{q(u)}^{q(u)} = \IG_{q(u)}$. Fix $f \in \Sigma_c(\GG;\EE)$. By \cref{lemma: bisection span}, we can write $f = \sum_{D \in F} f_D$, where $F$ is a finite collection of open bisections of $\GG$ such that for each $D \in F$, $f_D \in \Sigma_c(\GG;\EE)$ and $q(\supp(f_D)) \subseteq D$. Choose open neighbourhoods $\{ V_D \subseteq \GGo : D \in F \}$ of $q(u)$ as in the proof of \cite[Lemma~3.3(b)]{BNRSW2016}, so that $(V_D \, D \, V_D) \medcap q(\supp(f_D)) \subseteq \IG$ for each $D \in F$. Let $V \coloneqq \medcap_{D \in F} V_D$. Then $V$ is an open neighbourhood of $q(u)$ contained in $\GGo$. Now use Urysohn's lemma to choose $b \in C_c(V)$ such that $b \ge 0$ and $b(q(u)) = \lv b \rv_\infty = 1$. Since $\GGo$ is a bisection of $\GG$, it follows that $\lv b \rv = \lv b \rv_r = \lv b \rv_\infty = 1$ (see, for instance, \cite[Corollary~9.3.4]{Sims2020}). For each $\varepsilon \in q^{-1}(\GGo)$, there are unique elements $v_\varepsilon \in \GGo$ and $z_\varepsilon \in \T$ such that $\varepsilon = i(v_\varepsilon, z_\varepsilon)$. Define $g\colon \IE \to \C$ by
\[
g(\varepsilon) \coloneqq \begin{cases}
z_\varepsilon \, b(v_\varepsilon) & \text{if } \varepsilon \in q^{-1}(\GGo) \\
0 & \text{if } \varepsilon \notin q^{-1}(\GGo).
\end{cases}
\]
It follows immediately from our choice of $b$ and construction of $g$ that $g(u) = b(q(u)) = 1$ and that $q(\supp(g)) \subseteq V \subseteq \GGo$. By \cite[Lemma~11.1.9]{Sims2020}, there is an isomorphism from $C_c(\GGo)$ to $D_0^\II \coloneqq \{ h \in \Sigma_c(\IG;\IE) : q(\supp(h)) \subseteq \GGo \}$ that maps $b$ to $g$, and thus $g \in \Sigma_c(\IG;\IE)$ and $g \ge 0$. By \cite[Theorem~11.1.11]{Sims2020}, this isomorphism extends to (isometric) isomorphisms of $C_0(\GGo)$ onto the full and reduced C*-completions of $D_0^\II$, and therefore, $\lv g \rv = \lv g \rv_r = \lv b \rv_\infty = 1$. Finally, for each $D \in F$, we have
\[
q(\supp(g f_D g)) \subseteq q(\supp(g)) \, q(\supp(f_D)) \, q(\supp(g)) \subseteq (VD\,V) \medcap q(\supp(f_D)) \subseteq \IG
\]
by construction, and since $f = \sum_{D \in F} f_D$, it follows that $\supp(g f g) \subseteq q^{-1}(\IG) = \IE$.
\end{proof}

The next result is an extension of \cite[Lemma~3.5]{BNRSW2016} to the twisted setting.

\begin{lemma} \label{lemma: C*(GG;Sigma) is C*(II;Sigma)-compressible}
Let $(\EE,i,q)$ be a twist over a Hausdorff \'etale groupoid $\GG$. Let
\[
\iota_r\colon C_r^*(\IG;\IE) \to C_r^*(\GG;\EE) \quad \text{ and } \quad \iota\colon C^*(\IG;\IE) \to C^*(\GG;\EE)
\]
be the homomorphisms of \cref{prop: injective homo}, and define
\[
M_r \coloneqq \iota_r\big(C_r^*(\IG;\IE)\big) \quad \text{ and } \quad M \coloneqq \iota\big(C^*(\IG;\IE)\big).
\]
Suppose that $u \in \EEo$ satisfies $\EE_u^u = \IE_u$.
\begin{enumerate}[label=(\alph*)]
\item \label{item: compressible reduced} Fix $\epsilon > 0$ and $a \in C_r^*(\GG;\EE)$. There exist $b, c \in M_r$ satisfying $b \ge 0$, $\lv b \rv_r = 1$, $\lv bab - c \rv_r < \epsilon$, and $\varphi_r(b) = 1$ for every state $\varphi_r$ of $M_r$ such that $\varphi_r \circ \iota_r$ factors through $C_r^*\big(\GG_{q(u)}^{q(u)}; \EE_u^u\big)$. If $a$ is positive, then $c$ can be taken to be positive.
\item Fix $\epsilon > 0$ and $a \in C^*(\GG;\EE)$. There exist $b, c \in M$ satisfying $b \ge 0$, $\lv b \rv = 1$, $\lv bab - c \rv < \epsilon$, and $\varphi(b) = 1$ for every state $\varphi$ of $M$ such that $\varphi \circ \iota$ factors through $C^*\big(\GG_{q(u)}^{q(u)}; \EE_u^u\big)$. If $a$ is positive, then $c$ can be taken to be positive.
\end{enumerate}
\end{lemma}

\begin{proof}
Both parts follow from \cref{cor: quotient map on unit space algebra} in the same way, and so we will only prove part~\cref{item: compressible reduced}. First observe that since $\EE_u^u = \IE_u$, we have $\GG_{q(u)}^{q(u)} = \IG_{q(u)}$. Let
\[
Q_u^r\colon C_r^*(\IG;\IE) \to C_r^*(\IG_{q(u)}; \IE_u) = C_r^*\big(\GG_{q(u)}^{q(u)}; \EE_u^u\big)
\]
be the surjective $*$-homomorphism of \cref{cor: quotient map on unit space algebra}\cref{item: reduced quotient map on unit space algebra}. Since $\Sigma_c(\GG;\EE)$ is dense in $C_r^*(\GG;\EE)$, we can choose $f \in \Sigma_c(\GG;\EE)$ such that $\lv a - f \rv_r < \epsilon$. If $a$ is positive, a standard C*-algebraic argument (see, for instance, \cite[Lemma~5.3.10]{Armstrong2019}) shows that $f$ can also be taken to be positive. Use \cref{lemma: element for compression} to choose $g \in \Sigma_c(\IG;\IE)$ satisfying $g \ge 0$, $q(\supp(g)) \subseteq \GGo$, $\lv g \rv = \lv g \rv_r = g(u) = 1$, and $\supp(gfg) \subseteq \IE$. Define $b \coloneqq \iota_r(g)$ and $c \coloneqq bfb$. Then $\supp(c) \subseteq \IE$, $b, c \in M_r$, $b \ge 0$, and $b(u) = g(u) = 1$. If $f \ge 0$, then it follows that $c \ge 0$. Since $q(\supp(b)) \subseteq \GGo$, \cite[Theorem~11.1.11]{Sims2020} implies that $\lv b \rv = \lv b \rv_r$. Moreover, since \cref{prop: injective homo} implies that $\iota_r$ is isometric, we have $\lv b \rv = \lv b \rv_r = \lv \iota_r(g) \rv_r = \lv g \rv_r = 1$, and hence
\[
\lv bab - c \rv_r = \lv bab - bfb \rv_r \le \lv b \rv_r^2 \, \lv a - f \rv_r = \lv a - f \rv_r < \varepsilon.
\]
Suppose that $\varphi_r$ is a state of $M_r$ such that $\varphi_r \circ \iota_r$ factors through $C_r^*\big(\GG_{q(u)}^{q(u)}; \EE_u^u\big)$. Then there is a state $\psi_r$ of $C_r^*\big(\GG_{q(u)}^{q(u)}; \EE_u^u\big)$ such that $\varphi_r \circ \iota_r = \psi_r \circ Q_u^r$. By \cref{cor: quotient map on unit space algebra}\cref{item: reduced quotient map on unit space algebra}, we have $Q_u^r(g) = \delta^\T_u$, which is the identity element of $C_r^*\big(\GG_{q(u)}^{q(u)}; \EE_u^u\big)$ defined in \cref{prop: delta^T_varepsilon}. Thus, since $\psi_r$ is unital, we have $\varphi_r(b) = \varphi_r(\iota_r(g)) = \psi_r(Q_u^r(g)) = \psi_r(\delta^\T_u) = 1$.
\end{proof}

\begin{proof}[Proof of \cref{prop: unique extensions of states of M and M_r}]
Both parts follow from \cref{lemma: C*(GG;Sigma) is C*(II;Sigma)-compressible} in the same way, and so we will only prove part~\cref{item: unique extensions reduced}. Suppose that $\varphi_r$ is a state of $M_r$ such that $\varphi_r \circ \iota_r$ factors through $C_r^*\big(\GG_{q(u)}^{q(u)};\EE_u^u\big)$. Recalling the terminology defined in \cref{def: compressible}, $C_r^*(\GG;\EE)$ is $M_r$-compressible modulo $\varphi_r$ by \cref{lemma: C*(GG;Sigma) is C*(II;Sigma)-compressible}\cref{item: compressible reduced}, and so by \cref{thm: unique state extensions}, $\varphi_r$ has a unique state extension to $C_r^*(\GG;\EE)$.
\end{proof}

We need the following two additional results in order to prove \cref{thm: uniqueness theorem}.

\begin{lemma} \label{lemma: pure state composition}
Let $(\EE,i,q)$ be a twist over a Hausdorff \'etale groupoid $\GG$. Suppose that $u \in \EEo$ satisfies $\EE_u^u = \IE_u$, and let $Q_u^r\colon C_r^*(\IG;\IE) \to C_r^*(\IG_{q(u)}; \IE_u) = C_r^*\big(\GG_{q(u)}^{q(u)}; \EE_u^u\big)$ be the surjective $*$-homomorphism of \cref{cor: quotient map on unit space algebra}\cref{item: reduced quotient map on unit space algebra}. Let $\iota_r\colon C_r^*(\IG;\IE) \to C_r^*(\GG;\EE)$ be the injective homomorphism of \cref{prop: injective homo}, and define $M_r \coloneqq \iota_r\big(C_r^*(\IG;\IE)\big)$. Let $\phi$ be a state of $C_r^*\big(\GG_{q(u)}^{q(u)}; \EE_u^u\big)$, and define $\psi\colon M_r \to \C$ by $\psi(\iota_r(a)) \coloneqq \phi(Q_u^r(a))$. Then $\psi$ is a state of $M_r$, and $\psi \circ \iota_r$ is a state of $C_r^*(\IG;\IE)$ that factors through $C_r^*\big(\GG_{q(u)}^{q(u)}; \EE_u^u\big)$. If $\phi$ is a pure state, then $\psi$ and $\psi \circ \iota_r$ are also pure states.
\end{lemma}

\begin{proof}
Since $Q_u^r$ is a C*-homomorphism and $\phi$ is a state, it is clear that $\psi$ and $\psi \circ \iota_r$ are positive bounded linear functionals. To see that $\psi$ and $\psi \circ \iota_r$ are states, we must show that $\lv \psi \rv = \lv \psi \circ \iota_r \rv = 1$. Using \cref{lemma: element for compression}, we can find $g \in \Sigma_c(\IG;\IE)$ satisfying $q(\supp(g)) \subseteq \GGo$ and $g(u) = 1$, and hence \cref{cor: quotient map on unit space algebra}\cref{item: reduced quotient map on unit space algebra} implies that $Q_u^r(g)$ is the identity element of $C_r^*\big(\GG_{q(u)}^{q(u)}; \EE_u^u\big)$. Therefore, since $\phi$ is a state of $C_r^*\big(\GG_{q(u)}^{q(u)}; \EE_u^u\big)$, we have
\[
\lv \psi \rv \ge \lv \psi \circ \iota_r \rv \ge \lav \psi(\iota_r(g)) \rav = \lav \phi(Q_u^r(g)) \rav = 1,
\]
and so $\lv \psi \rv = \lv \psi \circ \iota_r \rv = 1$. Thus $\psi$ is a state of $M_r$, and $\psi \circ \iota_r = \phi \circ Q_u^r$ is a state of $C_r^*(\IG;\IE)$ that factors through $C_r^*\big(\GG_{q(u)}^{q(u)}; \EE_u^u\big)$.

Suppose now that $\phi$ is a pure state. We claim that $\psi \circ \iota_r$ is a pure state. To see this, suppose that $\psi_1$ and $\psi_2$ are states of $C_r^*(\IG;\IE)$ such that $\psi \circ \iota_r = t\psi_1 + (1-t)\psi_2$ for some $t \in (0,1)$. We must show that $\psi \circ \iota_r = \psi_1 = \psi_2$. We first claim that $\psi_1$ and $\psi_2$ factor through $C_r^*\big(\GG_{q(u)}^{q(u)}; \EE_u^u\big)$. To see this, it suffices to show that $\ker(Q_u^r) \subseteq \ker(\psi_1) \medcap \ker(\psi_2)$. For all $a \in \ker(Q_u^r)$ such that $a \ge 0$, we have $\psi_1(a), \psi_2(a) \ge 0$, and
\[
0 \le t\psi_1(a) + (1-t)\psi_2(a) = \psi(\iota_r(a)) = \phi(Q_u^r(a)) = \phi(0) = 0,
\]
and hence $\psi_1(a) = 0 = \psi_2(a)$. Since the kernel of $Q_u^r$ is a C*-algebra, it is spanned by its positive elements, and so we deduce that $\ker(Q_u^r) \subseteq \ker(\psi_1) \medcap \ker(\psi_2)$. Hence there exist states $\phi_1$ and $\phi_2$ of $C_r^*\big(\GG_{q(u)}^{q(u)}; \EE_u^u\big)$ such that $\psi_1 = \phi_1 \circ Q_u^r$ and $\psi_2 = \phi_2 \circ Q_u^r$. We have
\begin{align*}
\phi \circ Q_u^r = \psi \circ \iota_r &= t\psi_1 + (1-t)\psi_2 \\
&= t(\phi_1 \circ Q_u^r) + (1-t)(\phi_2 \circ Q_u^r) = (t\phi_1 + (1-t)\phi_2) \circ Q_u^r. \numberthis \label{eqn: pure states factoring through}
\end{align*}
Since $Q_u^r$ is surjective, we deduce from \cref{eqn: pure states factoring through} that $\phi = t\phi_1 + (1-t)\phi_2$. Hence $\phi = \phi_1 = \phi_2$, because $\phi$ is a pure state. Therefore,
\[
\psi_1 = \phi_1 \circ Q_u^r = \phi \circ Q_u^r = \psi \circ \iota_r \quad \text{ and } \quad \psi_2 = \phi_2 \circ Q_u^r = \phi \circ Q_u^r = \psi \circ \iota_r ,
\]
and so $\psi \circ \iota_r$ is a pure state. A similar argument shows that $\psi$ is also a pure state.
\end{proof}

Before we present the next result, we recall from \cite[Proposition~4.3]{Renault2008} the existence of the faithful conditional expectations $\Phi_r\colon C_r^*(\GG;\EE) \to C_r^*\big(\GGo;q^{-1}(\GGo)\big) \cong C_0(\GGo)$ and $\Phi_r^\II\colon C_r^*(\IG;\IE) \to C_r^*\big(\GGo;q^{-1}(\GGo)\big) \cong C_0(\GGo)$ extending restriction of functions.

\begin{lemma} \label{lemma: ev map circ cond exp}
Let $(\EE,i,q)$ be a twist over a Hausdorff \'etale groupoid $\GG$. Suppose that $u \in \EEo$ satisfies $\EE_u^u = \IE_u$. Let $\iota_r\colon C_r^*(\IG;\IE) \to C_r^*(\GG;\EE)$ be the injective homomorphism of \cref{prop: injective homo}, and define $M_r \coloneqq \iota_r\big(C_r^*(\IG;\IE)\big)$. Let $\ev_u$ be the evaluation map $f \mapsto f(u)$ on $C_r^*\big(\GGo;q^{-1}(\GGo)\big) \cong C_0(\GGo)$. Then
\begin{enumerate}[label=(\alph*)]
\item \label{item: ev state factors through} $\ev_u \circ\, \Phi_r^\II$ is a state of $C_r^*(\IG;\IE)$ that factors through $C_r^*\big(\GG_{q(u)}^{q(u)}; \EE_u^u\big)$;
\item \label{item: ev state equality} $\ev_u \circ\, (\Phi_r\restr{M_r})$ is a state of $M_r$ that satisfies $\ev_u \circ\, (\Phi_r\restr{M_r}) \circ \iota_r = \ev_u \circ\, \Phi_r^\II$; and
\item \label{item: ev unique state extension} $\ev_u \circ\, \Phi_r$ is the unique state extension of $\ev_u \circ\, (\Phi_r\restr{M_r})$ to $C_r^*(\GG;\EE)$.
\end{enumerate}
\end{lemma}

\begin{proof}
It is clear that $\ev_u \circ\, \Phi_r^\II$, $\ev_u \circ\, (\Phi_r\restr{M_r})$, and $\ev_u \circ\, \Phi_r$ are positive bounded linear functionals since they are composed of positive bounded linear maps. To see that they are states, use \cref{lemma: element for compression} to find $g \in \Sigma_c(\IG;\IE)$ such that $g(u) = 1$. Then
\[
\big\lav \Phi_r\restr{M_r}(\iota_r(g))(u) \big\rav = \left\lav \Phi_r^\II(g)(u) \right\rav = \lav g(u) \rav = 1,
\]
and it follows that
\[
\left\lv \ev_u \circ\, \Phi_r^\II \right\rv = \left\lv \ev_u \circ\, (\Phi_r\restr{M_r}) \right\rv = \left\lv \ev_u \circ\, \Phi_r \right\rv = 1.
\]
Thus $\ev_u \circ\, \Phi_r^\II$, $\ev_u \circ\, (\Phi_r\restr{M_r})$, and $\ev_u \circ\, \Phi_r$ are states.

Since $\ev_u \circ\, (\Phi_r\restr{M_r}) \circ \iota_r$ and $\ev_u \circ\, \Phi_r^\II$ agree on $\Sigma_c(\IG;\IE)$, which is dense in $C_r^*(\IG;\IE)$, it follows that $\ev_u \circ\, (\Phi_r\restr{M_r}) \circ \iota_r = \ev_u \circ\, \Phi_r^\II$. Thus part~\cref{item: ev state equality} holds.

For part~\cref{item: ev state factors through}, define $\HH \coloneqq \IG_{q(u)}$ and $\EE_\HH \coloneqq q^{-1}(\HH) = \IE_u$. Let $Q_u^r\colon C_r^*(\IG;\IE) \to C_r^*(\HH;\EE_\HH) = C_r^*\big(\GG_{q(u)}^{q(u)}; \EE_u^u\big)$ be the surjective $*$-homomorphism of \cref{cor: quotient map on unit space algebra}\cref{item: reduced quotient map on unit space algebra}, and let $\Phi_r^\HH\colon C_r^*\big(\HH;\EE_\HH\big) \to C_r^*\big(\HHo;q^{-1}(\HHo)\big) \cong C_0(\HHo)$ be the conditional expectation extending restriction of functions. To see that $\ev_u \circ\, \Phi_r^\II$ factors through $C_r^*(\HH;\EE_\HH)$, we will find a state $\phi_u$ of $C_r^*(\HH;\EE_\HH)$ such that $\phi_u \circ Q_u^r = \ev_u \circ\, \Phi_r^\II$. We have $\HHo = \{q(u)\}$, and by a similar argument to the one above, $\phi_u \coloneqq \ev_u \circ\, \Phi_r^\HH$ is a state of $C_r^*\big(\HH;\EE_\HH\big)$. For all $f \in \Sigma_c(\IG;\IE)$, we have $(\phi_u \circ Q_u^r)(f) = \phi_u\big(f\restr{\IE_u}\big) = f(u) = (\ev_u \circ\, \Phi_r^\II)(f)$. Since $\Sigma_c(\IG;\IE)$ is dense in $C_r^*(\IG;\IE)$, it follows that $\phi_u \circ Q_u^r = \ev_u \circ\, \Phi_r^\II$, as required.

We conclude by proving part~\cref{item: ev unique state extension}. By parts~\cref{item: ev state factors through,item: ev state equality}, $\ev_u \circ\, (\Phi_r\restr{M_r}) \circ \iota_r = \ev_u \circ\, \Phi_r^\II$ is a state of $C_r^*(\IG;\IE)$ that factors through $C_r^*\big(\GG_{q(u)}^{q(u)}; \EE_u^u\big)$. Therefore, \cref{prop: unique extensions of states of M and M_r}\cref{item: unique extensions reduced} implies that $\ev_u \circ\, (\Phi_r\restr{M_r})$ extends uniquely to a state of $C_r^*(\GG;\EE)$. Thus, since $\ev_u \circ\, \Phi_r$ is an extension of $\ev_u \circ\, (\Phi_r\restr{M_r})$ to $C_r^*(\GG;\EE)$, it must be the unique state extension.
\end{proof}

\begin{proof}[Proof of \cref{thm: uniqueness theorem}]
Since $\iota_r$ is injective, it is clear that if $\Psi$ is injective, then so is the homomorphism $\Psi \circ \iota_r$. We prove the converse. For this, suppose that $\Psi \circ \iota_r$ is injective. Then $\Psi$ is injective on the subalgebra $M_r$ of $C_r^*(\GG;\EE)$. Let
\[
\XE \coloneqq \{ u \in \EEo : \EE_u^u = \IE_u \} \quad \text{ and } \quad \XG \coloneqq q(\XE) = \{ x \in \GGo : \GG_x^x = \IG_x \}.
\]
For each $u \in \XE$, let $\SS_u$ be the collection of pure states $\varphi$ of $M_r$ such that $\varphi \circ \iota_r$ is a pure state of $C_r^*(\IG;\IE)$ that factors through $C_r^*\big(\GG_{q(u)}^{q(u)}; \EE_u^u\big)$. Define $\SS \coloneqq \medcup_{u \in \XE} \, \SS_u$. By \cref{prop: unique extensions of states of M and M_r}\cref{item: unique extensions reduced}, each $\varphi \in \SS$ extends uniquely to a state $\overline{\varphi}$ of $C_r^*(\GG;\EE)$. For each $\varphi \in \SS$, let $\pi_{\overline{\varphi}}$ be the GNS representation of $C_r^*(\GG;\EE)$ associated to $\overline{\varphi}$. To see that $\Psi$ is injective on $C_r^*(\GG;\EE)$, it suffices by \cite[Theorem~3.2]{BNRSW2016}, to show that $\pi_\SS \coloneqq \bigoplus_{\varphi \in \SS} \pi_{\overline{\varphi}}$ is faithful on $C_r^*(\GG;\EE)$. For this, fix $a \in C_r^*(\GG;\EE)$ such that $\pi_\SS(a) = 0$. Then $\pi_{\overline{\varphi}}(a) = 0$ for every $\varphi \in \SS$. Let $\Phi_r\colon C_r^*(\GG;\EE) \to C_r^*\big(\GGo;q^{-1}(\GGo)\big)$ be the faithful conditional expectation extending restriction of functions. To see that $a = 0$, it suffices to show that $\Phi_r(a^*a) = 0$, because $\Phi_r$ is faithful. Suppose, for contradiction, that $\Phi_r(a^*a) \ne 0$. Then $\Phi_r(a^*a) > 0$, because $\Phi_r$ is positive. Let $\Xi_a \in C_0(\GGo)$ be the image of $\Phi_r(a^*a)$ under the isomorphism from $C_r^*\big(\GGo;q^{-1}(\GGo)\big)$ to $C_0(\GGo)$ given in \cite[Theorem~11.1.11]{Sims2020}. Then $\Xi_a > 0$, since $\Phi_r(a^*a) > 0$. Let $Y_a \coloneqq \Xi_a^{-1}\big((0,\infty)\big)$. Since $\Xi_a$ is continuous, $Y_a$ is an open subset of $\GGo$. Thus, since $\XG$ is dense in $\GGo$ by \cite[Lemma~3.3(a)]{BNRSW2016}, we have $Y_a \medcap \XG \ne \varnothing$. Choose $x \in Y_a \medcap \XG$, and let $u \coloneqq i(x,1)$. Then $u \in (q\restr{\EEo})^{-1}(\XG) = \XE$, and $\Phi_r(a^*a)(u) = \Xi_a(x) > 0$. Fix $\epsilon > 0$ such that
\begin{equation} \label{eqn: positive cond exp}
\Phi_r(a^*a)(u) > \epsilon.
\end{equation}
Since $\EE_u^u = \IE_u$ and $a^*a \ge 0$, we know by \cref{lemma: C*(GG;Sigma) is C*(II;Sigma)-compressible}\cref{item: compressible reduced} that there exist $b, c \in M_r$ such that $b, c \ge 0$,
\begin{equation} \label{eqn: compression of a^*a}
\lv b a^* a b - c \rv_r < \frac{\epsilon}{2},
\end{equation}
and $\varphi(b) = \lv b \rv_r = 1$ for every (not necessarily pure) state $\varphi$ of $M_r$ such that $\varphi \circ \iota_r$ factors through $C_r^*\big(\GG_{q(u)}^{q(u)}; \EE_u^u\big)$. Let $(e_\lambda)_{\lambda \in \Lambda}$ be an approximate identity for $C_r^*(\GG;\EE)$, and for each $\varphi \in \SS_u$, let $N_{\overline{\varphi}} \coloneqq \{ f \in C_r^*(\GG;\EE) : \overline{\varphi}(f^*f) = 0 \}$ be the null space for $\overline{\varphi}$. Since the GNS representation $\pi_{\overline{\varphi}}$ satisfies $\pi_{\overline{\varphi}}(a) = 0$, we have $\pi_{\overline{\varphi}}(ba^*ab) = 0$, and so by the GNS construction, we have
\begin{equation} \label{eqn: state with compression}
\overline{\varphi}(ba^*ab) = \lim_{\lambda \in \Lambda} \big( \pi_{\overline{\varphi}}(ba^*ab)(e_\lambda + N_{\overline{\varphi}}) \mid e_\lambda + N_{\overline{\varphi}} \big) = 0, \ \text{ for each } \varphi \in \SS_u.
\end{equation}
Together, \cref{eqn: compression of a^*a,eqn: state with compression} imply that for all $\varphi \in \SS_u$,
\begin{equation} \label{eqn: state almost kills element}
\lav \varphi(c) \rav = \lav \overline{\varphi}(c) \rav \le \left\lav \overline{\varphi}(c) - \overline{\varphi}(ba^*ab) \right\rav + \left\lav \overline{\varphi}(ba^*ab) \right\rav \le \lv c - b a^* a b \rv_r < \frac{\epsilon}{2}.
\end{equation}
Let $Q_u^r\colon C_r^*(\IG;\IE) \to C_r^*(\IG_{q(u)}; \IE_u) = C_r^*\big(\GG_{q(u)}^{q(u)}; \EE_u^u\big)$ be the surjective $*$-homomorphism of \cref{cor: quotient map on unit space algebra}\cref{item: reduced quotient map on unit space algebra}. Since $c$ is a positive element of $M_r$, $Q_u^r(\iota_r^{-1}(c))$ is a positive element of $C_r^*\big(\GG_{q(u)}^{q(u)}; \EE_u^u\big)$, and so by \cite[Proposition~II.6.3.3]{Blackadar2006}, there is a pure state $\phi$ of $C_r^*\big(\GG_{q(u)}^{q(u)}; \EE_u^u\big)$ satisfying
\begin{equation} \label{eqn: state for element to achieve norm}
\left\lav \phi\big(Q_u^r(\iota_r^{-1}(c))\big) \right\rav = \left\lv Q_u^r(\iota_r^{-1}(c)) \right\rv_r.
\end{equation}
Define $\psi\colon M_r \to \C$ by $\psi(\iota_r(h)) \coloneqq \phi(Q_u^r(h))$. By \cref{lemma: pure state composition}, $\psi$ is a pure state of $M_r$ and $\psi \circ \iota_r$ is a pure state of $C_r^*(\IG;\IE)$ that factors through $C_r^*\big(\GG_{q(u)}^{q(u)}; \EE_u^u\big)$. Hence $\psi \in \SS_u$, and so \cref{eqn: state for element to achieve norm,eqn: state almost kills element} imply that
\begin{equation} \label{eqn: quotient map almost kills element}
\left\lv Q_u^r(\iota_r^{-1}(c)) \right\rv_r = \left\lav \phi\big(Q_u^r(\iota_r^{-1}(c))\big) \right\rav = \lav \psi(c) \rav < \frac{\epsilon}{2}.
\end{equation}
Let $\ev_u$ be the evaluation map $f \mapsto f(u)$ on $C_r^*\big(\GGo;q^{-1}(\GGo)\big)$, and let $\rho_u \coloneqq \ev_u \circ\, (\Phi_r\restr{M_r})$. By parts~\cref{item: ev state factors through,item: ev state equality} of \cref{lemma: ev map circ cond exp}, $\rho_u$ is a state of $M_r$ and $\rho_u \circ \iota_r$ is a state of $C_r^*(\IG;\IE)$ that factors through $C_r^*\big(\GG_{q(u)}^{q(u)}; \EE_u^u\big)$. Hence there is a state $\kappa_u$ of $C_r^*\big(\GG_{q(u)}^{q(u)}; \EE_u^u\big)$ such that $\rho_u \circ \iota_r = \kappa_u \circ Q_u^r$. Thus, using \cref{eqn: quotient map almost kills element} for the final inequality, we obtain
\begin{equation} \label{eqn: evaluation of cond exp almost kills element}
\left\lav \rho_u(c) \right\rav = \left\lav \kappa_u\big(Q_u^r(\iota_r^{-1}(c))\big) \right\rav \le \left\lv Q_u^r(\iota_r^{-1}(c)) \right\rv_r < \frac{\epsilon}{2}.
\end{equation}
By \cref{lemma: ev map circ cond exp}\cref{item: ev unique state extension}, $\overline{\rho}_u \coloneqq \ev_u \circ\, \Phi_r$ is the unique state extension of $\rho_u$ to $C_r^*(\GG;\EE)$. By our choice of $b$, we have $\overline{\rho}_u(b) = \rho_u(b) = \lv b \rv_r = 1$. Thus, recalling the notation defined in \cref{notation: state sets}, we have $b \in \UU_{\overline{\rho}_u}$, and so \cref{rem: U_phi subseteq M_phi} implies that $b \in \MM_{\overline{\rho}_u}$. Therefore,
\begin{equation} \label{eqn: pulling elements out of evaluation of cond exp}
\overline{\rho}_u(ba^*ab) = \overline{\rho}_u(b) \, \overline{\rho}_u(a^*a) \, \overline{\rho}_u(b) = \overline{\rho}_u(a^*a).
\end{equation}
Using \cref{eqn: pulling elements out of evaluation of cond exp} for the second equality, we obtain
\begin{align*}
\Phi_r(a^*a)(u) = \left\lav \overline{\rho}_u(a^*a) \right\rav &= \left\lav \overline{\rho}_u(ba^*ab) \right\rav \\
&\le \left\lav \overline{\rho}_u(ba^*ab) - \overline{\rho}_u(c) \right\rav + \left\lav \overline{\rho}_u(c) \right\rav \le \lv b a^*a b - c \rv_r + \left\lav \rho_u(c) \right\rav\!. \numberthis \label{eqn: Phi_r(a*a) inequality}
\end{align*}
Together, \cref{eqn: Phi_r(a*a) inequality,eqn: compression of a^*a,eqn: evaluation of cond exp almost kills element} imply that
\[
\Phi_r(a^*a)(u) \le \lv b a^*a b - c \rv_r + \left\lav \rho_u(c) \right\rav < \frac{\epsilon}{2} + \frac{\epsilon}{2} = \epsilon,
\]
which contradicts the inequality~\labelcref{eqn: positive cond exp}. Thus, we deduce that $\Phi_r(a^*a) = 0$, which completes the proof.
\end{proof}

We conclude with a characterisation of simplicity of reduced twisted C*-algebras of effective Hausdorff \'etale groupoids. This result is presumably well known (for instance, if $\GG$ is minimal and effective, then \cite[Theorem~7.26]{KM2021} implies that $C_r^*(\GG;\EE)$ is simple), but we present a proof as an example of an application of \cref{thm: uniqueness theorem}.

Note that if the groupoid $\GG$ is not effective, then characterising simplicity of $C_r^*(\GG;\EE)$ in terms of $\GG$ and $\EE$ is a much harder problem; see \cite[Remark~8.3]{KPS2015TAMS}.

\begin{cor} \label{cor: simplicity}
Let $(\EE,i,q)$ be a twist over an effective Hausdorff \'etale groupoid $\GG$. Then $C_r^*(\GG;\EE)$ is simple if and only if $\GG$ is minimal.
\end{cor}

\begin{proof}
Suppose that $\GG$ is not minimal. Then there exists a nonempty proper closed invariant subset $K$ of $\GGo$. Define $\GK \coloneqq s^{-1}(K)$. Since $K$ is closed and invariant, $\GK$ is a closed \'etale subgroupoid of $\GG$ with unit space $\GK^{(0)} = K$. Hence \cref{lemma: subgroupoid twists} implies that $q^{-1}(\GK)$ is a twist over $\GK$. The restriction map $f \mapsto f\restr{q^{-1}(\GK)}$ is a $*$\nobreakdash-homomorphism from $\Sigma_c(\GG;\EE)$ to $\Sigma_c\big(\GK;q^{-1}(\GK)\big)$, and so it extends to a C*-homomorphism $\res{K}\colon C_r^*(\GG;\EE) \to C_r^*\big(\GK;q^{-1}(\GK)\big)$. Since $\GK$ is nonempty, $\res{K}$ is not the zero map, and since $\GK \ne \GG$, $\res{K}$ is not injective. Hence $\ker(\res{K})$ is a nonzero proper ideal of $C_r^*(\GG;\EE)$, and so $C_r^*(\GG;\EE)$ is not simple.

For the converse, suppose that $\GG$ is minimal. Let $D_r$ denote the completion of the set $D_0 \coloneqq \left\{ f \in \Sigma_c(\GG;\EE) : q(\supp(f)) \subseteq \GGo \right\}$ with respect to the reduced norm. Since $\GG$ is effective, we have $\IG = \GGo$ and $\IE = q^{-1}(\GGo)$. Let $\iota_r\colon C_r^*\big(\GGo;q^{-1}(\GGo)\big) \to C_r^*(\GG;\EE)$ be the injective homomorphism of \cref{prop: injective homo}. Then $\iota_r\big(C_r^*\big(\GGo;q^{-1}(\GGo)\big)\big) = D_r$.

Let $I$ be a nonzero ideal of $C_r^*(\GG;\EE)$. Then there is a C*-homomorphism $\Psi$ of $C_r^*(\GG;\EE)$ such that $I = \ker(\Psi)$. Since $I$ is nonzero, $\Psi$ is not injective, and hence \cref{thm: uniqueness theorem} implies that $\Psi \circ \iota_r$ is not injective either. Thus $J \coloneqq \ker(\Psi \circ \iota_r)$ is a nonzero ideal of $C_r^*\big(\GGo;q^{-1}(\GGo)\big)$, and we have $\iota_r(J) = I \cap D_r$. To see that $C_r^*(\GG;\EE)$ is simple, we must show that $I = C_r^*(\GG;\EE)$. We know by \cite[Theorem~5.2]{Renault2008} that $D_r$ contains an approximate identity for $C_r^*(\GG;\EE)$ (see also, \cite[Proposition~11.1.14]{Sims2020}), and so it suffices to show that $\iota_r(J) = D_r$, because then $D_r \subseteq I$, and it follows that $I = C_r^*(\GG;\EE)$.

Recall from \cite[Theorem~11.1.11]{Sims2020} that there is an isomorphism $\Upsilon\colon C_0(\GGo) \to D_r$ such that $\Upsilon(f)(i(x,z)) = z \, f(x)$ for all $f \in C_0(\GGo)$ and $(x,z) \in \GGo \times \T$. Define
\[
F \coloneqq \{ x \in \GGo : f(x) = 0 \text{ for all } f \in \Upsilon^{-1}(\iota_r(J)) \}.
\]
Since $\Upsilon^{-1}(\iota_r(J))$ is a nonzero ideal of $C_0(\GGo)$, $F$ is a proper closed subset of $\GGo$, and
\[
\Upsilon^{-1}(\iota_r(J)) = \{ f \in C_0(\GGo) : f(x) = 0 \text{ for all } x \in F \}.
\]
To see that $\iota_r(J) = D_r$, we will prove the equivalent statement that $\Upsilon^{-1}(\iota_r(J)) = C_0(\GGo)$. Suppose that $F \ne \varnothing$. We will derive a contradiction by showing that $F = \GGo$. Since $\GG$ is minimal, the only closed invariant subsets of $\GGo$ are $\varnothing$ and $\GGo$ itself, and so it suffices to show that $F$ is invariant. For this, fix $x \in F$, and suppose that $\gamma \in \GG$ satisfies $s(\gamma) = x$. We must show that $r(\gamma) \in F$. For this, fix $f \in \Upsilon^{-1}(\iota_r(J))$. We must show that $f(r(\gamma)) = 0$. Use \cref{lemma: CLS bisections and units} to choose a local trivialisation $(B_\alpha, P_\alpha, \phi_{P_\alpha})_{\alpha \in \GG}$ of $\EE$ such that each $B_\alpha$ is a bisection of $\GG$. Use Urysohn's lemma to choose $h \in C_c(\GG)$ such that $\supp(h) \subseteq B_\gamma$ and $h(\gamma) = 1$. Recall from \cref{lemma: T-action properties}\cref{item: t_alpha cts} that for each $\varepsilon \in q^{-1}(B_\gamma)$, there is a unique $z_\varepsilon \in \T$ such that $\varepsilon = \phi_{P_\gamma}(q(\varepsilon),z_\varepsilon)$, and the map $\varepsilon \mapsto z_\varepsilon$ is continuous on $q^{-1}(B_\gamma)$. Thus $\varepsilon \mapsto z_\varepsilon \, h(q(\varepsilon))$ is a continuous map from $q^{-1}(B_\gamma)$ to $\C$. Define $g\colon \EE \to \C$ by
\[
g(\varepsilon) \coloneqq \begin{cases}
z_\varepsilon \, h(q(\varepsilon)) & \ \text{if } \varepsilon \in q^{-1}(B_\gamma) \\
0 & \ \text{if } \varepsilon \notin q^{-1}(B_\gamma).
\end{cases}
\]
Since $P_\gamma(\gamma) = \phi_{P_\gamma}(\gamma,1)$, we have $g(P_\gamma(\gamma)) = h(\gamma) = 1$. By a similar argument to the one used in the proof of \cref{lemma: j_varepsilon} to show that $j_\varepsilon \in \Sigma_c(\IG;\IE)$, we see that $g \in \Sigma_c(\GG;\EE)$. Since $\supp(g) \subseteq q^{-1}(B_\gamma)$ and $\supp(\Upsilon(f)) \subseteq q^{-1}(\GGo)$, it follows from \cref{eqn: convolution as a finite sum} and the fact that $B_\gamma$ is a bisection that
\[
q\big(\!\supp\!\big(g^* \, \Upsilon(f) \, g\big)\big) \subseteq B_\gamma^{-1} \, \GGo \, B_\gamma = s(B_\gamma) \subseteq \GGo,
\]
and thus $g^* \, \Upsilon(f) \, g \in D_r$. Since $\Upsilon(f)$ is an element of the ideal $I$, it follows that $g^* \, \Upsilon(f) \, g \in I \cap D_r = \iota_r(J)$. Therefore, since $x \in F$, we have
\begin{equation} \label{eqn: conjugation gives 0}
\big(g^* \, \Upsilon(f) \, g\big)(i(x,1)) = \Upsilon^{-1}\big(g^* \, \Upsilon(f) \, g\big)(x) = 0.
\end{equation}
Observe that $i(x,1) = s\big(P_\gamma(\gamma)\big) = P_\gamma(\gamma)^{-1} \, i(r(\gamma),1) \, P_\gamma(\gamma)$. Thus, since $q(\supp(g))$ is contained in the bisection $B_\gamma$, \cref{eqn: convolution as a finite sum} implies that
\begin{equation} \label{eqn: conjugation simplification}
\big(g^* \, \Upsilon(f) \, g\big)(i(x,1)) = g^*\big(P_\gamma(\gamma)^{-1}\big) \, \Upsilon(f)\big(i(r(\gamma),1)\big) \, g\big(P_\gamma(\gamma)\big) = f(r(\gamma)).
\end{equation}
Together, \cref{eqn: conjugation gives 0,eqn: conjugation simplification} imply that $f(r(\gamma)) = 0$, and so $F$ is invariant. Thus $F = \GGo$, which is a contradiction, because $\Upsilon^{-1}(\iota_r(J))$ is a nonzero ideal of $C_0(\GGo)$. Therefore, we must have $F = \varnothing$, and so $\Upsilon^{-1}(\iota_r(J)) = C_0(\GGo)$, as required.
\end{proof}

\vspace{2ex}
\raggedbottom
\bibliographystyle{amsplain}
\makeatletter\renewcommand\@biblabel[1]{[#1]}\makeatother
\bibliography{Armstrong1_references}
\vspace{4ex}

\end{document}